\documentclass [10pt]{article}

\usepackage{amssymb}
\usepackage{amsfonts}
\usepackage{graphicx}
\usepackage{amsmath}
\usepackage{amsmath,amsfonts,amssymb}
\usepackage{dsfont}
\usepackage{graphics}
\usepackage{subcaption}
\usepackage{mathtools}
\usepackage{commath}
\usepackage[colorlinks=true]{hyperref}

\usepackage{color}

\numberwithin{equation}{section}

\def\[{\left[}
\def\]{\right]}
\def\({\left(}
\def\){\right)}

\newcommand{\R}{\mathbb{R}}

\newtheorem{theorem}{Theorem}[section]

\newtheorem{assumption}[theorem]{Assumption}

\newtheorem{claim}[theorem]{Claim}

\newtheorem{corollary}[theorem]{Corollary}

\newtheorem{lemma}[theorem]{Lemma}

\newtheorem{proposition}[theorem]{Proposition}
\newtheorem{remark}[theorem]{Remark}

\newenvironment{proof}[1][Proof]{\noindent\textit{#1.} }{\hfill \rule{0.5em}{0.5em}}

\renewcommand{\eqref}[1]{(\ref{#1})}

\newcommand{\al}{\alpha}
\newcommand{\be}{\beta}
\renewcommand{\R}{\mathbb{R}}
\newcommand{\La}{\Lambda}
\newcommand{\la}{\lambda}
\newcommand{\ga}{\gamma}

\newcommand{\de}{\delta}

\newcommand{\ep}{\epsilon}
\newcommand{\1}{\mathds{1}}

\numberwithin{equation}{section}
\numberwithin{figure}{section}
\allowdisplaybreaks

\begin{document}

\title{\textbf{Quantifying the threshold phenomena for propagation in nonlocal diffusion equations}}

\author{
	\textsc{Matthieu Alfaro$^{(a)}$, Arnaud Ducrot$^{(b)}$, Hao Kang$^{(b)}$}\\
	{\small \textit{$^{(a)}$  Universit\'e de Rouen Normandie, CNRS, LMRS, Saint-Etienne-du-Rouvray, France}}\\
	{\small \textit{$^{(b)}$ Normandie Univ, UNIHAVRE, LMAH, FR-CNRS-3335, ISCN, 76600 Le Havre, France}} \\
}

\maketitle

\tableofcontents

\begin{abstract}
We are interested in the threshold phenomena for propagation in nonlocal diffusion equations with some compactly supported initial data. In the so-called bistable and ignition cases, we provide the first quantitative estimates for such phenomena. The outcomes dramatically depend on the tails of the dispersal kernel and can take a large variety of different forms.  The strategy is to combine sharp estimates of the tails of the sum of i.i.d. random variables (coming, in particular, from large deviation theory) and the construction of accurate sub- and super-solutions. 

\medskip
	

\end{abstract}

\section{Introduction}

In this work we consider the solution $u=u(t,x)$ of the one-dimensional integro-differential equation
\begin{equation}\label{nonlocal}
\partial_t u (t,x)=\int_\R J(y)\left[u(t,x-y)-u(t,x)\right]d y+f\left(u(t,x)\right),\quad t>0,\, x\in\R,
\end{equation}
supplemented with some compactly supported initial data $u(0,x)=u_0(x)$. The function  $J:\R\to \R$ is a nonnegative dispersal kernel of total mass 1 allowing, in population dynamics models, to take into account long distance dispersal events. The nonlinearity $f=f(u)$ encodes the demographic assumptions and,  in this manuscript, we  focus on the case where it is of the {\it bistable} or of the {\it ignition} type. Precise assumptions will  be given later on. Our concern is to understand the {\it threshold phenomenon} for propagation. As far as we know, we provide the first quantitative estimates for such phenomenon in \eqref{nonlocal}. 
 
\medskip

The solutions of \eqref{nonlocal} share some properties with the ones of the local diffusion equation 
\begin{equation}\label{local}
\partial_t u (t,x)=\partial_{xx}u(t,x)+f\left(u(t,x)\right),\quad t>0,\,x\in\R.
\end{equation}
In particular both problems exhibit a \textit{threshold phenomenon}, meaning that \lq\lq small" initial data lead to extinction, whereas \lq\lq large" initial data lead to propagation. In the local diffusion case \eqref{local}, we refer to \cite{Aro-Wei-78} for such a property, while the {\it sharp} threshold phenomenon was more recently investigated \cite{zlatovs2006sharp}, \cite{matano2010}, \cite{Pol-11}, \cite{Mur-Zho-13, Mur-Zho-17}, through different technics. On the other hand, the nonlocal diffusion case \eqref{nonlocal} is more delicate, see subsection \ref{ss:towards} below, but some progresses were recently achieved  \cite{alfaro2017fujita}, \cite{berestycki2017non}, \cite{limlong}.

Very recently, the authors (including the  first two authors of the present work)  of \cite{alfaro2020quantitative}  have provided a sharp {\it quantitative} estimate of the threshold phenomenon in \eqref{local}. To be more precise, let $f=f(u)$ be of the ignition or bistable type between $0$ and $1$, with threshold $\theta\in (0,1)$, and \lq\lq $1$ more stable than $0$'' (see Assumption \ref{ass:f} below). Then, for any $\ep \in (0,1-\theta)$ there are two lengths $0<L_\ep^{\rm ext}<L_\ep^{\rm prop}<+\infty$ such that the solution $u_L^\ep=u_L^\ep(t,x)$ of \eqref{local} starting from
$$
\phi_L^\ep(x):=(\theta+\ep) \1 _{(-L,L)}(x),
$$
satisfies
\begin{eqnarray*}
\lim_{t\to+\infty} u_L^\ep (t,x)=\begin{cases} 0 & \text{ uniformly for $x\in\R$, if $0<L<L_\ep^{\rm ext}$},\\
1 & \text{ locally uniformly for $x\in\R$, if $L>L_\ep^{\rm prop}$}.
\end{cases}
\end{eqnarray*}
Among others, the study \cite{alfaro2020quantitative} provides some estimates of the threshold values $L_\ep^{\rm ext}$ and $L_\ep^{\rm prop}$ as $\ep \to 0^+$, namely
$$
0<\liminf_{\ep \to 0+} \frac{L_\ep^{\rm ext}}{-\ln \ep }\leq \limsup_{\ep\to 0^+} \frac{L_\ep^{\rm prop}}{-\ln \ep}<+\infty.
$$ 
As far as we know, for integro-differential equation such as \eqref{nonlocal} there is no such quantitative estimates of this threshold phenomenon in the literature. The goal of the present work is to fill this gap by first deriving sufficient conditions for the existence of similar critical lengths $L_\ep^{\rm ext}$ ad $L_\ep^{\rm prop}$ for suitable levels $\ep>0$, and secondly to provide estimates of these quantities in some asymptotic regime, namely when the height $\theta+\ep$ of the initial data tends to $\theta$.

Note that the analysis in \cite{alfaro2020quantitative} is  based on the explicit formula of the heat kernel. For the nonlocal diffusion case, the expression  of the \lq\lq corresponding heat kernel'', see \eqref{def-K}--\eqref{def-psi}, is much more complicated and estimating its behaviour is far from obvious. However, refined estimates for the tails of the kernel $J$, as well as for its $n-$folds convolution $J^{*(n)}$, can be obtained for a large variety of kernels. Such information, mainly coming from the probability theory, combined with the construction of accurate sub-  and super-solutions will allow us to derive some precise estimates of the threshold phenomenon for the nonlocal diffusion equation \eqref{nonlocal}.

It turns outs that these estimates of the threshold phenomenon strongly depend on the decay rate at infinity for the kernel $J$, or {\it tails} of $J$. For instance, when the kernel $J$ has an exponential decay then the critical lengths (when exist) both behave like for the local problem \eqref{local}. On the other hand, for kernels with {\it heavy} tails (algebraic decay, Weibull-like tails, etc.), such estimates become more complicated and can take a large variety of different forms.

\medskip

Let us first present the assumptions on the dispersal kernel $J$ and the nonlinearity $f$.

\begin{assumption}[Dispersal kernel]\label{ass:kernel} The kernel $J:\R\to \R$ satisfies the following.
	\begin{itemize}
		\item [(i)] $J$ is nonnegative, even and satisfies $\int_\R J(x)dx=1$;  
		\item [(ii)] $J$ is nonincreasing in $(0,+\infty)$. 
	\end{itemize} 
\end{assumption}
	
\begin{assumption}[Expansion form]\label{ass:kernel-bis}	
	The Fourier transform of $J$ has an expansion
	\begin{equation}
	\label{J-Fourier-ass}
	\widehat J(\xi)=1-a\vert \xi\vert ^{\be}+o(\vert \xi \vert ^{\be}), \quad \text{ as } \xi \to 0,
	\end{equation}
	for some $0<\be\leq 2$ and $a>0$.
\end{assumption}

As observed and proved in \cite{chasseigne2006asymptotic}, expansion \eqref{J-Fourier-ass} plays a crucial role in the behavior of the linear equation $\partial _t u=J*u-u$, but also  in some nonlinear phenomena, such as the Fujita blow up phenomenon and the hair trigger effect \cite{alfaro2017fujita}.

Notice that expansion \eqref{J-Fourier-ass} contains some information on the tails of $J$. Indeed, for kernels which have a finite second momentum, namely $m_2(J):=\int_{\R}x^2J(x)d x<+\infty$,  expansion \eqref{J-Fourier-ass} holds true with $\be =2$, as can be seen in \cite[Chapter 2]{durrett2019probability} among others. In particular, this is the case for kernels which are compactly supported, exponentially bounded, or which decrease like  $\mathcal O\left(\frac{1}{\vert x\vert^{3+\ep}}\right)$ with $\ep>0$. On the other hand, when $m_2(J)=+\infty$ then more general expansions are possible. For example, for algebraic tails satisfying 
\begin{equation}
\label{heavytails}
J(x)\sim \frac{C}{\vert x\vert ^{\al}} \quad \text{ as }\vert x\vert \to \infty, \quad\text{ with } 1<\al< 3,
\end{equation}
then \eqref{J-Fourier-ass} holds true with $\be= \al- 1\in(0,2)$. This fact is related to the stable laws of index $\be\in(0,2)$ in probability theory, and a proof can be found in \cite[Chapter 2, subsection 2.7]{durrett2019probability}. In particular it contains the case of the Cauchy law $J(x)=\frac{1/\pi}{1+x^2}$, for which
$$
\widehat J(\xi)=e^{-\vert\xi\vert}=1-\vert\xi\vert +o(\vert \xi\vert), \quad  \text{ as } \xi \to 0,
$$
and $\be =1$, despite the nonexistence of the first moment $m_1(J):=\int \vert x\vert J(x)d x$.

\medskip

Throughout this note the reaction term arising in \eqref{nonlocal} will satisfy the following set of hypothesis.

\begin{assumption}[Nonlinearity]\label{ass:f}
	The function $f: \R\to\R$ is Lipschitz continuous. There is a threshold $\theta\in (0, 1)$ such that
	\begin{equation}\label{0-theta-1}
	f(u)=0,\;\forall u\in\{0, \theta, 1\},
	\end{equation} 
	and 
	\begin{equation}\label{bist-ignition}
	f(u)>0,\;\forall u\in(\theta, 1), \quad\text{and}\quad
	\begin{cases}
	f(u)<0,\;\forall u\in (0, \theta), \quad(\text{Bistable Case}),\\
	\text{or}\\
	f(u)=0,\;\forall u\in (0, \theta), \quad(\text{Ignition Case}).
	\end{cases}
	\end{equation}
	In the bistable case, we further require 
	\begin{equation}\label{positive}
	\int_{0}^{1}f(s) ds>0,
	\end{equation} 
Moreover, in both cases, 	 we require that there are $r^->0$ and $\delta\in (\theta,1)$ such that
\begin{equation}\label{r-moins}
f(u)\geq r^-\left(u-\theta\right),\;\forall u\in [0,\delta].
\end{equation}
\end{assumption}	
Notice that for $r>0$, the usual cubic bistable nonlinearity
$$ 
f(u)=ru(u-\theta)(1-u)\1_{(0, 1)}(u) 
$$
satisfies the above assumptions as soon as $\theta<\frac{1}{2}$, and so does the ignition nonlinearity
$$ 
f(u)=r(u-\theta)(1-u)\1_{(\theta, 1)}(u). 
$$

\medskip

The organization of this work is as follows. In section \ref{s:basic}, we list some basic facts and key lemmas, in particular some estimates for the tails of $n-$folds convolution kernel (coming from sum of i.i.d. random variables in probability theory).
 In section \ref{main results} we state our main results including our quantitative estimates on the threshold phenomenon \lq\lq extinction vs. propagation''. In sections \ref{extinction} and \ref{non-extinction} we inquire on extinction and non-extinction phenomena, respectively, in some related toy models. 
 Next, in section \ref{quantitative}, we build on these preliminary results to prove our main results. Finally, in section \ref{propagation} we establish  a sufficient condition for \lq\lq propagation to occur'', which is another one of our main results.
 
\section{Preliminaries}\label{s:basic}

In this section, after presenting some notations and basic facts, we collect  some estimates for the tails of $J^{*(i)}$ ($i\geq 1$) coming from the probability theory.

\subsection{Notations and basic linear facts}\label{ss:linear}

If $f\in L^1(\R)$, we define its Fourier transform $\mathcal{F}(f) = \widehat{f}$ and its inverse Fourier transform $\mathcal{F}^{-1} (f)$ by
\begin{equation*}
\widehat{f}(\xi):=
\int_{\R} e^{-i\xi x} f(x) dx, \quad \mathcal{F}^{-1}(f)(x)
:=\int_{\R} e^{ix\xi} f(\xi) d\xi. 
\end{equation*}
With this definition, we have, for $f$, $g\in L^1(\R)$,
$$
\widehat{f*g}=\widehat f \, \widehat g,
$$
and $f=\frac{1}{2\pi}\mathcal F ^{-1}(\mathcal F(f))$ if $f$, $\mathcal F (f) \in L ^1(\R)$. 

\bigskip

Formally applying the Fourier transform to equation
\begin{equation}\label{nonlocal-heat}
\partial _t u=J*u-u
\end{equation}
yields
$$
\frac{d}{dt}\widehat{u}(t, \xi) = \widehat{u}(t, \xi) (\widehat{J}(\xi)-1),
$$
where $\xi$ is seen as a parameter and which is solved as 
\begin{equation*} 
\widehat{u}(t, \xi) = e^{t(\widehat{J}(\xi)-1)} \widehat{u}_0(\xi).
\end{equation*}
Applying the inverse Fourier transform, we see that the fundamental solution writes as
\begin{equation}\label{def-K}
K(t, \cdot)=e^{-t}\de_0+\psi(t, \cdot),
\end{equation}
where 
\begin{equation}\label{def-psi} 
\psi(t,\cdot) := e^{-t} \sum_{i=1}^{\infty} \frac{t^i}{i!} J^{\ast(i)}, 
\end{equation}
where $\de_0$ is the Dirac mass at $0$ and $J^{\ast(i)}:=J\ast\cdots\ast J$ is the convolution of $J$ with itself $i-1$ times.

Hence the unique bounded solution to $\partial _tu=J*u-u$ with initial data $u_0\in L^\infty(\R)$ is given by
$$
u(t, x) = K(t, \cdot)\ast u_0(x)=e^{-t}u_0(x) + \psi(t, \cdot) \ast u_0(x). 
$$ 
Obviously, though the convolution of a Dirac mass by a $L^\infty$ function is not pointwise well defined, we let $\de_0\ast u_0=u_0$. Moreover, from the normal convergence of the series 
$\sum_{i=1}^{\infty} \frac{t^i}{i!} J^{\ast(i)}$ 
and $\sum_{i=1}^{\infty} \frac{t^{i-1}}{(i-1)!} J^{\ast(i)}$ in $C([0, T], L^1(\R))$, we deduce that the function $t\in[0, \infty)\to\psi(t, \cdot)\in L^1(\R)$ is of class $C^1$ and that
\begin{equation}\label{partial}
\partial_t \psi(t, x) = J\ast \psi (t, \cdot)(x) - \psi (t, x) + e^{-t} J(x).
\end{equation}
Notice also that 
\begin{equation}\label{1-e^t}
\int_{\R} \psi (t, x) dx = 1-e^{-t}, \quad \forall t\ge0. 
\end{equation}

\subsection{Some key estimates on the tails of $J^{*(i)}$ ($i\geq 1$)}\label{ss:linear2}

Observe that if $J\in L^1(\R)$ satisfies Assumption \ref{ass:kernel-bis} then so does $J^{*(i)}$, with $i\geq 1$, but with the expansion
\begin{equation}
\label{J-Fourier-ass-i}
\widehat {J^{*(i)}} (\xi)=1-i\,a\vert \xi\vert ^{\be}+o(\vert \xi \vert ^{\be}), \quad \text{ as } \xi \to 0.
\end{equation}
Notice also that $\widehat{J}$ is a real function since $J$ is even. In the sequel, for $i\geq 1$ and $L>0$, we denote
\begin{equation}
\label{def-R-i-L}
R_i(L):=\int _{\vert x\vert \geq L} J^{*(i)}(x)dx
\end{equation}
the tails of $J^{*(i)}$. The expansion \eqref{J-Fourier-ass-i} on the low frequencies of the Fourier transform $\widehat {J^{*(i)}}$  is very related to the tails of $J^{*(i)}$. The following first  bound is provided
by  \cite[Chapter 2 subsection 2.3.c (3.5)]{durrett2019probability}.

\begin{lemma}\label{lem:Durret}
	Let Assumption \ref{ass:kernel}-(i) be satisfied. For any $i\geq 1$ and $L>0$, one has
	\begin{equation}
	R_i(L)\leq \frac{L}{2}\int_{-\frac{2}{L}}^{\frac{2}{L}} \left[1-\left(\widehat J(\xi)\right)^i\right] d\xi.
	\end{equation}
\end{lemma}

As $L\to+\infty$, we have  the following result quoted from \cite[Theorem 5]{pitman1968}.

\begin{lemma}\label{lem:tails} 
	Let Assumptions \ref{ass:kernel}-(i) and \ref{ass:kernel-bis} hold.
	\begin{itemize}
		\item [(i)]
		If $\be=2$ then, for any $i\geq 1$, one has
		\begin{equation}\label{tails-beta-equal-2}
		\int _0^L  uR_i(u) du \sim i\,a, \quad \text{ as } L\to +\infty.
		\end{equation}
	  \item [(ii)]
		If $0<\be <2$ then there is $C=C(\be)>0$ such that, for all $i\geq 1$,
		\begin{equation}\label{tails-beta-lower-2}
		R_i(L) \sim i \,\frac{aC}{L ^{\be}}, \quad \text{ as } L\to +\infty.
		\end{equation}
	\end{itemize}
\end{lemma}

When $0<\beta <2$, the asymptotic behavior of $R_i(L)$ is thus very precisely described by Assumption \ref{ass:kernel-bis}. On the other hand, when $\beta=2$, Assumption \ref{ass:kernel-bis} is not enough to capture the asymptotic behavior of $R_i(L)$: precise estimates, as needed in this work, strongly depend on the decay rate of $J$ at infinity. Below we discuss such estimates for different types of kernels.

The following result deals with exponentially bounded kernels and  is known, in probability theory,  as the large deviations Cram\'er theorem (adapted to our situation). We refer the reader to \cite[section 2.2]{Dembo-Zeitouni} for more details and proofs. The uniformity of the limit \eqref{cramer} as stated below is a consequence of \cite{hoglund}, see  Appendix \ref{s:appendix}.

\begin{lemma}\label{lem:expo}
	Let Assumption \ref{ass:kernel}-(i) be satisfied and assume further
\begin{equation}\label{expo-decay}
\exists \lambda _0>0,\;\;\int_\R e^{\lambda _0 x} J(x)d x<+\infty.
\end{equation}	
	Define the logarithmic moment generating function $\La:\R\to (-\infty,+\infty]$ as 
	$$
	\La(\lambda):=\ln \left(\int_\R e^{\lambda x}J(x)d x\right),
	$$
	and its Fenchel-Legendre transform as
	$$
	\La ^*(x):=\sup_{\lambda \in \R} ( \lambda x -\La(\lambda)).
	$$
	
	Then, $\La ^*$ is nondecreasing on $(0,+\infty)$. Also, there exists $0<L_1$ such that $\La^*(L)<+\infty$ for all $L\in [0,L_1]$ and, for all $0<L_0<L_1$, one has
	\begin{equation}\label{cramer}
	\lim_{i\to+\infty}\frac{1}{i}\ln R_i(iL)	=-\La ^*(L), \quad \text{ uniformly for $L\in [L_0,L_1]$}.
	\end{equation}
    Moreover we also have for any $L>0$, $i\geq 1$ and $\lambda \in \R$,
    \begin{equation}\label{esti-expo}
    R_i(iL)\leq e^{-i \La^*(L)}\leq   e^{-i(\lambda L- \La(\lambda))}.
    \end{equation}
\end{lemma}

\medskip

When the kernel $J$ is not exponentially bounded (but still $\beta=2$), there is a large variety of possible tails. The following examples are typical, see e.g. \cite{mikosch1998large}, and will be considered in the following.

\medskip

\noindent {\bf Regularly varying tails $RV(\al)$}: 
\begin{equation}\label{tail-RV}
R_1(L)=L^{-\al} S(L),\quad \forall L>0,
\end{equation}
for $\al>2$ and where $S:(0,+\infty)\to(0,+\infty)$ is a slowly varying function (that is, for all $\nu>0$,  $S(\nu x)\sim S(x)$  as $x\to+\infty$).

\medskip

\noindent {\bf Lognormal-type tails  $LN(\ga, \lambda, \rho)$:}
\begin{equation}\label{tail-LN}
R_1(L) \sim cL^{\rho}e^{-\la \ln^\ga L},\quad \text{ as } L\to+\infty,
\end{equation}
for some $\ga>1$, $\la>0$, $\rho \in \R$ and appropriate constant $c=c(\rho, \ga)$. 

\medskip

\noindent {\bf Weibull-like tails $WE(\al,\la,\rho)$:}
\begin{equation}\label{tail-WE}
R_1(L)\sim cL^{\rho}e^{-\la L^\al},\quad \text{ as } L\to+\infty,
\end{equation}
for some $0<\al<1$, $\la>0$, $\rho \in \R$ and appropriate constant $c=c(\rho, \al)$. 

\medskip

For these three type classes of kernels, refined estimates of $R_i(L)$ when $i\gg 1$ and $L\gg 1$ are known in the literature: the next lemma is taken from \cite[Proposition 3.1]{mikosch1998large}.

\begin{lemma}\label{LDH}
	For the three classes of kernels above, assume $J$ is normalized by $m_2(J)=1$. Then define the threshold sequence $d_n$ accordingly to  
	\begin{center}
		\begin{tabular}{|l|l|}
			\hline
			 \qquad Distribution & \qquad $d_n$\\
			\hline 
			$RV(\al)$, \quad\qquad  $\al>2$ & $n^{1/2}\ln^{1/2} n  $\\
			$LN(\ga,\la,\rho)$, \quad $1<\ga\leq 2$ & $n^{1/2}\ln^{\ga/2}n$\\
			$LN(\ga,\la,\rho)$, \quad $2<\ga$ & $n^{1/2}\ln^{\ga-1}n$\\
			$WE(\al,\la,\rho)$, \; $0<\al<1$ & $n^{1/(2-2\al)}$\\
			\hline
		\end{tabular}\\
		\medskip
	Table 1: The threshold sequences $d_n$.
	\end{center}
Then, for any sequence $\gamma_n\gg d_n$, one has
	\begin{equation}\label{conclusion-table1}
	\lim_{i\to+\infty} \sup_{L\geq \gamma_i}\left\vert  \frac{R_i(L)}{iR_1(L)}-1\right \vert=0.
	\end{equation}
\end{lemma}

\medskip

Last, for some specific forms of Weibull-like tails satisfying ($0<\alpha <1$)
\begin{equation}
\label{tail-weibull-zero}
R_1(L)\sim  c e^{-L^\al}, \quad  \text{ as }L\to+\infty,
\end{equation}
more refined estimates are known:  according to \cite[(2.32)]{Nagaev1979} there exists some constant $C>0$ such that, for all $L>0$ and $i\geq 1$,
\begin{equation}\label{WEbe=0}
R_i(L)\leq C\left[\exp\left(-\frac{L^2}{20 i}\right)+iR_1\left(\frac L 2\right)\right].
\end{equation}

\section{Main results}\label{main results}

For $\ep\in(0, 1-\theta]$ and $L>0$, we consider the family of initial data $\phi_L^\ep$ given by
\begin{equation}\label{initial}
\phi_L^\ep(x):=(\theta+\ep)\1_{(-L,L)}(x),\quad x\in\R,
\end{equation}
wherein $\1_A$ denotes the characteristic function of the set $A$. We denote by $u_L^\ep=u_L^\ep(t,x)$ the solution of \eqref{nonlocal} starting from the initial datum $\phi_L^\ep$. 

The results presented in this section are concerned with the derivation of asymptotic estimates  ($\ep\ll 1$)  of the size $L$ such that the solution $u_L^\ep(t,x)$ goes extinct or propagates at large times. We split our main results into two parts. The first subsection is related to extinction while the second is concerned with propagation.

\subsection{Extinction results}\label{ss:ext}

Let us recall that, for $i\geq 1$ and $L>0$, $R_i(L)$ denotes the tails of $J^{*(i)}$ as defined in \eqref{def-R-i-L}. Our general extinction criterion reads as follows. 

\begin{theorem}[Extinction]\label{ext}
Let Assumptions \ref{ass:kernel} and \ref{ass:f} be satisfied. Let $\ep>0$ be fixed and set
\begin{equation}\label{r^+}
r^+:=\sup_{u\in (\theta,1]} \frac{f(u)}{u-\theta}>0.
\end{equation}
Assume $T>0$ and $L>0$ are such that
\begin{equation}\label{criterion}
\theta e^{-(r^++1)T}\sum_{i=1}^{+\infty} \frac{T^i}{i!}R_i(L)\geq \ep.
\end{equation}
Then the function $u_L^\ep$ is going to extinction at large times, namely
	\begin{equation*}
	\lim_{t\to+\infty}\sup_{x\in\R}u_L^\ep(t, x)=0.
	\end{equation*}
\end{theorem} 

The above result has various implications exploring the asymptotic size ensuring the extinction of the solution. More precisely, for each $\ep\in (0,1-\theta]$, define the interval
$$
\mathcal L_\ep^{\rm ext}:=\left\{L>0:\;\lim_{t\to+\infty}\sup_{x\in\R}u_L^\ep(t, x)=0\right\},
$$
as well as
\begin{equation}\label{DEF-Lext}
L_\ep^{\rm ext}:=\begin{cases}\sup\mathcal L_\ep^{\rm ext}\in (0,\infty]  & \text{ if }\; \mathcal L_\ep^{\rm ext}\neq \emptyset,\\
0 & \text{ if } \;\mathcal L_\ep^{\rm ext}=\emptyset.
\end{cases}
\end{equation}
The next corollaries are concerned with a lower bounded of $L_\ep^{\rm ext}$ as $\ep\to 0^+ $.

We start with the case where the kernel $J$ has a slow decay at infinity, meaning that $0<\beta<2$ in the expansion \eqref{J-Fourier-ass}.

\begin{corollary}[Asymptotic extinction criterion, $0<\be<2$]\label{asmptotic1}
Let Assumptions \ref{ass:kernel}, \ref{ass:kernel-bis} with $0<\be <2$, and \ref{ass:f} be satisfied.  Then
\begin{equation*}
\liminf_{\ep \to 0^+} \; \ep^{\frac{1}{\beta}}L^{\rm ext}_\ep\geq C^-:= \left(\theta aC e^{-1}/r^+ \right)^{1/\be}>0,
\end{equation*}
where $a>0$ and $C=C(\beta)>0$ come from the asymptotic expansion \eqref{J-Fourier-ass} and Lemma \ref{lem:tails}-$(ii)$, respectively. 
\end{corollary}

Next, we explore a similar question when $\beta=2$. We start with the case of exponentially bounded kernel.

\begin{corollary}[Asymptotic extinction criterion for exponentially bounded kernels]\label{asmptotic2} Let Assumptions \ref{ass:kernel} and \ref{ass:f} be satisfied and assume further the  exponential decay \eqref{expo-decay}. Then 
\begin{equation*}
\liminf_{\ep \to 0^+}\; \frac{L_\ep^{\rm ext}}{\ln \frac{1}{\ep}}\ge C^->0,
\end{equation*}
for some $C^-=C^-(r^+,J)>0$. 
\end{corollary}

We pursue the $\beta=2$ case by presenting a series of classical kernels presented in subsection \ref{ss:linear2}.

\begin{corollary}[Asymptotic extinction criterion, some examples with $\be=2$]\label{ext:examples}
Let Assumptions \ref{ass:kernel} and \ref{ass:f} be satisfied. Then the following estimates hold.
	\begin{itemize}
	\item [(i)] Assume that there are $c_2>c_1>0$, $\alpha>2$ and $x_0\in \R$ such that $ \frac{c_1}{\vert x \vert^{1+\alpha}}\leq J(x) \leq  \frac{c_2}{\vert x \vert^{1+\alpha}}$ for all $\vert x\vert \geq \vert x_0\vert $. Then 
	$$
	\liminf_{\ep \to 0^+}\; \ep^{\frac {1} {\al}}L_\ep^{\rm ext}\ge C^->0,
	$$
    for some $C^-=C^-(\theta,r^+,T,c_1,\alpha)>0$.
    
    \item[(ii)]	 Assume that $J$ has  regularly varying tails $RV(\al)$ with $\alpha>2$. Then 
    $$
    \liminf_{\ep \to 0^+}\; \ep^{\frac {1} {\al'}}L_\ep^{\rm ext}\ge 1,
    $$
    for any given $\alpha'>\alpha$.
    
    \item [(iii)]  Assume that $J$ has lognormal-type tails $LN(\gamma,\lambda,\rho)$ with $\gamma>1$, $\lambda>0$ and $\rho \in \R$. Then 
    $$
	\liminf_{\ep \to 0^+}\; \frac{L_\ep^{\rm ext}}{e^{\left(\frac{1}{\la '}\ln \frac 1 \ep\right)^{\frac 1 \ga}}}\ge 1,
	$$
	for any given $\la '>\la$.
		    
    \item [(iv)]  Assume that $J$ has  Weibull-like tails $WE(\alpha,\lambda,\rho)$ with $0<\al <1$, $\lambda>0$ and $\rho \in \R$. Then 
    $$
    \liminf_{\ep \to 0^+}\; \frac{L_\ep^{\rm ext}}{\left(\frac 1{\la '}\ln \frac 1 \ep\right)^{\frac1\al}}\ge 1,
    $$    
    for any given $\la '>\la$. Furthermore, if $\rho=0$, 
    $$
    \liminf_{\ep \to 0^+}\; \frac{L_\ep^{\rm ext}}{\left(\frac 1{\la }\ln \frac 1 \ep\right)^{\frac1\al}}\ge C^-,
    $$  
    for some $C^-=C^-(\theta,r^+,T,J)>0$.
\end{itemize}
\end{corollary}

\subsection{Propagation results}\label{ss:prop}

To state our propagation results, we need to assume that for any level $\theta+\ep>\theta$, there exists a size $L$ such that the solution $u_L^\ep$ of \eqref{nonlocal} with initial data \eqref{initial} propagates. This assumption will be largely discussed below.

\begin{assumption}[Propagation assumption]\label{ass:prop} Let Assumptions \ref{ass:kernel}-(i) and \ref{ass:f} hold.  We assume that for each $\ep\in (0,1-\theta]$ there exists $L>0$ large enough such that $u_L^\ep(t, x)$ propagates, in the sense that 
	$$
	\lim_{t\to+\infty} \; u_L^\ep(t,x)=1\; \text{ locally uniformly for on $\R$}.
	$$
\end{assumption}
Now using this and similarly as above, for each $\ep\in (0,1-\theta]$, we define the non-empty interval
$$
\mathcal L_\ep^{\rm prop}:=\left\{L>0:\;\lim_{t\to+\infty}u_L^\ep(t, x)=1\text{ locally uniformly on $\R$}\right\},
$$
as well as the quantity
\begin{equation}\label{DEF-Lprop}
L_\ep^{\rm prop}:=\inf\mathcal L_\ep^{\rm prop}\in [0,\infty).
\end{equation}

We now state our general propagation criterion (let us recall that $\delta\in(\theta,1)$ is coming from \eqref{r-moins}).

\begin{theorem}[Propagation]\label{theo-propagation}
Let Assumptions \ref{ass:kernel}-(i), \ref{ass:f} and \ref{ass:prop} holds. For any $\al\in (0,\delta-\theta)$ and $m\in (0,1)$
there exists $\ep_0>0$ small enough such that, for all $\ep\in (0,\ep_0)$, one has $L_\ep^{\rm prop}\le L$ where $L$ satisfies the following condition
\begin{equation}\label{Condition}
\frac{\ep}{2(\theta+\ep)}\geq e^{-T_\ep }\sum_{i=1}^{+\infty} \frac{T_\ep^i}{i!}R_i\left((1-m)L\right),\quad\text{ with } T_\ep=\frac{1}{r^-}\ln\frac{2\alpha}{\ep},
\end{equation}
 where we recall that $R_i(L)$ denotes the tails of $J^{*(i)}$ as defined in \eqref{def-R-i-L},
\end{theorem}

 In the following we derive an upper bound of  $L_\ep^{\rm prop}$ as $\ep\to 0^+$. As before, this asymptotic strongly depends on the tails of the kernels $J^{*(i)}$. 

\begin{corollary}[Asymptotic propagation criterion, $0<\beta<2$]\label{prop:corollary1}
	Let Assumptions \ref{ass:kernel}-(i), \ref{ass:kernel-bis} with $0<\be<2$, \ref{ass:f} and \ref{ass:prop} be satisfied. Then
	$$
	\limsup_{\ep\to0^+}\frac{L_\ep^{\rm prop}}{\left(\frac{1}{\ep}\ln\frac{1}{\ep}\right)^{\frac 1 \be}}\le C^+,
	$$
	for some $C^+=C^+(\theta,r^-,J)>0$.
\end{corollary}

Next, we explore a similar question when $\be=2$. For such cases with $\be=2$, we assume $r^-\in(0, 1)$, where $r^-$ is from \eqref{r-moins}. We start with the case of exponentially bounded kernel again. 

\begin{corollary}[Asymptotic propagation criterion for exponentially bounded kernels]\label{prop:corollary2}
	Let Assumptions \ref{ass:kernel}-(i), \ref{ass:f} and \ref{ass:prop} be satisfied. Assume further the  exponential decay \eqref{expo-decay} and $r^-\in(0, 1)$. Then
	$$ 
	\limsup_{\ep\to0^+}\frac{L_\ep^{\rm prop}}{\ln\frac 1 \ep}\le C^+,
	$$    
	for some $C^+=C^+(r^-,J)>0$.
\end{corollary}

We pursue the $\be=2$ case by presenting a series of classical kernels given in subsection \ref{ss:linear2}. 

\begin{corollary}[Asymptotic propagation criterion, some examples with $\be=2$]\label{prop:corollary3}
	Let Assumptions \ref{ass:kernel}-(i), \ref{ass:f} and \ref{ass:prop} be satisfied. In addition, assume $r^-\in(0, 1)$ and $m_2(J)=\int_{\R}x^2J(x)dx=1$. Then we have the following results.
\begin{itemize}	  
   \item [(i)] Assume that there are $c_2>c_1>0$, $\alpha>2$ and $x_0\in \R$ such that $\frac{c_1}{\vert x\vert^{1+\alpha}}\leq J(x)\leq \frac{c_2}{\vert x\vert ^{1+\alpha}}$ for all $\vert x\vert  \geq \vert x_0\vert$ then 
   $$
   \limsup_{\ep\to0^+}\frac{L_\ep^{\rm prop}}{\left(\frac{1}{\ep}\ln\frac{1}{\ep}\right)^{\frac 1 \alpha}}\le C^+,
   $$
   for some $C^+=C^+(\theta, r^-, c_2, \alpha)>0$.

   \item [(ii)] Assume that $J$ has regularly varying tails $RV(\al)$ with $\al>2$. Then 
   $$
   \limsup_{\ep\to0^+}\frac{L_\ep^{\rm prop}}{\left(\frac{1}{\ep}\ln\frac{1}{\ep}\right)^{\frac 1 {\tilde \al}}}\le 1,
   $$
   for any given $0<\tilde \alpha<\alpha$. 

   \item [(iii)] Assume that $J$ has lognormal-type tails $LN(\ga,\la,\rho)$ with $\ga>1$, $\la>0$, $\rho \in \R$. Then 
   $$
   \limsup_{\ep\to0^+}\frac{L_\ep^{\rm prop}}{\exp \left[\left(\frac{1}{\tilde \lambda} \ln \frac 1\ep\right)^\frac{1}{\ga}\right]}\le 1,
   $$
   for any given $0<\tilde \lambda<\lambda$.	
       	
   \item [(iv)] Assume that $J$ has Weibull-like tails $WE(\al,\la,\rho)$ with $0<\al<1$, $\la>0$, $\rho \in \R$. Then 
   $$
   \limsup_{\ep\to 0^+}\frac{L_\ep^{\rm prop}}{\left(\frac  1{\tilde \lambda}\ln \frac 1\ep\right)^{\frac 1\alpha}}\leq 1, \quad \text{ if } \quad 0<\al<\frac 2 3,
   $$
   for any given $0<\tilde \lambda<\lambda$, while	
   $$
   \limsup_{\ep\to 0^+}\frac{L_\ep^{\rm prop}}{\left(\ln \frac 1 \ep \right)^{\frac{1}{2(1-\tilde \alpha)}}}\leq 1, \quad \text{ if } \quad \frac 2 3\leq \alpha <1,
   $$
   for any given $0<\alpha<\tilde \alpha<1$. Furthermore, if $\rho=0$,
   $$
	\limsup_{\ep\to 0^+}\frac{L_\ep^{\rm prop}}{\left(\frac 1 \la \ln \frac 1 \ep \right)^{\frac{1}{\alpha}}}\leq C^+,
	$$
	for some $C^+=C^+(J)>0$.
\end{itemize}
\end{corollary}   

Note that when $\rho=0$ in $(iv)$, the assumption $m_2(J)=1$ is not necessary since we directly rely on estimate \eqref{WEbe=0}.

\subsection{Combining, summarizing  and commenting the above results}

Let the dispersal kernel satisfy Assumptions \ref{ass:kernel} and \ref{ass:kernel-bis}. Let the nonlinearity $f$ be of the bistable or ignition type in the sense of Assumption \ref{ass:f}. We considered the solution $u_L^\ep=u_L^\ep(t,x)$ to the one-dimensional reaction integro-differential equation \eqref{nonlocal} starting from the step function $\phi_L^\ep(x):=(\theta+\ep)\1_{(-L,L)}(x)$. Under the propagation Assumption \ref{ass:prop}, to be discussed in the next subsection, we investigated the values $0<L_\ep^{ext}\leq L_\ep^{prop}<+\infty$ in the asymptotic regime $\ep\to 0^+$, meaning that the height of the initial step function tends to the threshold value $\theta$ of nonlinearity $f$.  For a large variety of dispersal kernels, we proved some estimates stated in subsection \ref{ss:ext} and \ref{ss:prop} that we now combine and comment. 

\medskip

If $0<\beta <2$ in Assumption \ref{ass:kernel-bis}, meaning in particular that $m_2(J)=\int_{\R}x^2J(x)dx=+\infty$, then 
		$$
		0<\liminf_{\ep \to 0^+}\frac{L_\ep^{\rm ext}}{\left(\frac 1 \ep\right)^{\frac 1 \be}}\le \limsup_{\ep\to0^+}\frac{L_\ep^{\rm prop}}{\left(\frac{1}{\ep}\ln\frac{1}{\ep}\right)^{\frac 1 \be}}<+\infty, \quad \text{ (heavy tails)}.
		$$
In this  situation, the \lq\lq main term'' is of magnitude $\left(\frac 1\ep\right)^{\frac1\beta}$, but the lower and upper bounds coincide only up to a logarithmic term.

\medskip 

If $m_2(J)=\int_{\R}x^2J(x)dx<+\infty$, meaning $\beta=2$ in Assumption \ref{ass:kernel-bis}, we distinguish the following prototype situations under the additional assumption $r^-\in(0,1)$, that we believe to be only technical. 

If $J$ is exponentially bounded as in \eqref{expo-decay} then 
		$$
		0<\liminf_{\ep \to 0^+}\frac{L_\ep^{\rm ext}}{\ln \frac 1 \ep}\le \limsup_{\ep\to0^+}\frac{L_\ep^{\rm prop}}{\ln\frac 1 \ep}<+\infty, \quad \text{ (exponentially bounded tails)}.
		$$
		In this situation, the lower and upper bounds are {\it sharp} in the sense that they coincide up to a multiplicative constant.
		
		If $J$ has \lq\lq not heavy algebraic tails'', in the sense that there are 
$c_2>c_1>0$, $\alpha>2$ and $x_0\in \R$ such that
$$
\frac{c_1}{\vert x \vert^{1+\alpha}}\leq J(x) \leq  \frac{c_2}{\vert x \vert^{1+\alpha}}, \quad \text{ for all }  \vert x\vert \geq \vert x_0\vert,
$$
and moreover $m_2(J)=1$, then
		$$
		0<\liminf_{\ep \to 0^+}\frac{L_\ep^{\rm ext}}{\left(\frac 1 \ep\right)^{\frac 1 \al}}\le \limsup_{\ep\to0^+}\frac{L_\ep^{\rm prop}}{\left(\frac{1}{\ep}\ln\frac{1}{\ep}\right)^{\frac 1 \al}}<+\infty, \quad \text{ (not heavy algebraic tails)}.
		$$
		In this  situation, the main term is of magnitude $\left(\frac 1\ep\right)^{\frac1\alpha}$, but the lower and upper bounds coincide only up to a logarithmic term.

If $J$ has regularly varying tails $RV(\alpha)$ with $\alpha>2$ and $m_2(J)=1$, then for any $0<\tilde \alpha<\alpha<\alpha'$,
$$
	\limsup_{\ep\to0^+}\frac{L_\ep^{\rm prop}}{\left(\frac{1}{\ep}\ln\frac{1}{\ep}\right)^{\frac 1 {\tilde \al}}}\leq 1\leq \liminf_{\ep \to 0^+}\; \frac{L_\ep^{\rm ext}}{\left(\frac 1 \ep\right)^{\frac {1} {\al'}}} \quad \text{ (regularly varying tails)}.
$$
The situation is here less clear than the not heavy algebraic tails since, from the above, it is not clear that the main term is of magnitude  $\left(\frac 1\ep\right)^{\frac1\alpha}$.

If $J$ has lognormal-type tails $LN(\ga,\la,\rho)$ with $\ga>1$, $\la>0$, $\rho \in \R$ and $m_2(J)=1$, then for any $0<\tilde \lambda<\lambda<\lambda'$,
	$$
	    	\limsup_{\ep\to0^+}\frac{L_\ep^{\rm prop}}{\exp \left[\left(\frac{1}{\tilde \lambda} \ln \frac 1\ep\right)^\frac{1}{\ga}\right]}\le 1 \leq 
		\liminf_{\ep \to 0^+}\; \frac{L_\ep^{\rm ext}}{\exp\left[\left(\frac{1}{\la '} \ln \frac 1 \ep\right)^{\frac 1 \ga}\right]}, \quad \text{ (lognormal-type tails)}.
		$$
	
If $J$ has Weibull-like tails $WE(\al,\la,\rho=0)$ with $0<\al<1$, $\la>0$ (see above for the $\rho\neq 0$ case) then 
    $$
  0< \liminf_{\ep \to 0^+}\; \frac{L_\ep^{\rm ext}}{\left(\frac 1 \lambda \ln \frac 1 \ep\right)^{\frac 1\al}}\leq \limsup_{\ep\to 0^+}\frac{L_\ep^{\rm prop}}{\left(\frac 1 \lambda \ln \frac 1 \ep \right)^{\frac{1}{\alpha}}}<+\infty, \quad \text{ (Weibull-like tails $WE(\al, \la, \rho=0)$)}.
	$$
In this situation, the main term is of magnitude $\frac{1}{\left(\ln \frac 1 \ep\right)^{\frac 1 \alpha}}$ and the lower and upper bounds are sharp. 
	
\medskip
	
As a conclusion, our analysis reveals that the main term (or, at least, the approximate main term) driving the asymptotics of $L_\ep^{ext}$ and $L_\ep^{prop}$ is, in some sense, selected by the tails of $J$ and can take a large variety of different forms. 

\subsection{Towards Assumption \ref{ass:prop}}\label{ss:towards}

The propagation threshold in nonlocal diffusion problems is of fundamental importance and independent interest.  In this subsection, we derive  some sufficient conditions for the  propagation Assumption \ref{ass:prop} to hold. We place ourselves in the following {\it bistable} situation.

\begin{assumption}[Dispersal kernel and bistable nonlinearity]\label{TWA}  The dispersal kernel $J\in C^1(\R)$ is nonnegative, even, and satisfies $\int_\R J(x)dx=1$. Assume further $J'\in L^1(\R)$ and the existence of a finite first moment, namely 
	\begin{equation}\label{first_moment}
	\int_\R |y|J(y)dy<+\infty.
	\end{equation}
The nonlinearity $f\in C^1(\R)$ is bistable in the sense of \eqref{0-theta-1}---\eqref{bist-ignition} and further satisfies $f'(0)<0$, $f'(1)<0$.
\end{assumption}

A traveling wave solution $(c,U)$ for \eqref{nonlocal} is a speed $c\in \R$ and a profile $U=U(z)$ solving
\begin{equation}\label{TW}
\begin{cases}
J\ast U-U-cU'+f(U)=0,\quad \text{on }\,\R,\\
U(-\infty)=0, \quad U(+\infty)=1,
\end{cases}
\end{equation} 
meaning in particular that $u(t,x):=U(x+ct)$ solves \eqref{nonlocal}.
The existence, uniqueness and properties of traveling wave solutions in nonlocal diffusion equations with monostable, ignition or bistable nonlinearities has attracted a lot of attention, see e.g. \cite{Erm-Mac-93}, \cite{bates1997traveling}, \cite{Che-97}, \cite{coville2006uniqueness,coville2007travelling}, \cite{Cov-Dup-07}, \cite{Mel-Roq-Sir-14}, \cite{Gui-Hua-15}, \cite{alfaro2017propagation} and the references therein. 

Under Assumption \ref{TWA}, the existence of traveling waves is studied in the important work of Bates et al. \cite{bates1997traveling} (to which we refer for more precise statements): there exists a monotone solution, in the {\it weak} sense, to \eqref{TW}. The smoothness of the profile $U$ is a delicate issue. More precisely, assuming further a structure condition on $g(u):=u-f(u)$, see hypothesis (H3) in \cite{bates1997traveling}, some of the conclusions of \cite[Theorem 3.1]{bates1997traveling} are as follows: if $c\neq 0$ then the profile $U$ is smooth and the sign of the speed is that of $\int_0^1f(u)du$; on the other hand if $c=0$ there are some situations, see  \cite[end of Section 3]{bates1997traveling}, where the profile $U$ is discontinuous, and this can happen even if $\int_0^1f(u)du > 0$. Moreover, 
the $c=0$ case is ruled out by the non-existence of a {\it null-truncation} of $g(u)=u-f(u)$ in the sense of \cite[Theorem 1]{berestycki2017non}. In particular, a sufficient condition to enforce $c>0$ is to assume 
\begin{equation}
\label{hyp:g-monotone-prop}
u \mapsto u-f(u) \text{  is monotone and }  \int_0^1f(u)du > 0.
\end{equation}
This shows the relevance of our sufficient conditions for propagation to occur, which we now state.

\begin{theorem}[Propagation threshold]\label{theo-ass-prop}
Let Assumption \ref{TWA} be satisfied and assume $\int_0^1f(u)du>0$. Assume the existence of $(c,U)$  an increasing traveling wave solution to \eqref{nonlocal} with $c\neq 0$ (and thus $c>0$ from the above discussion). Then the conclusions of the propagation Assumption \ref{ass:prop} hold: for each $\ep\in (0,1-\theta]$ there exists $L>0$ large enough such that $u_L^\ep(t, x)$ propagates, in the sense that 
\begin{equation}\label{qqch}
	\lim_{t\to+\infty} \; u_L^\ep(t,x)=1\; \text{ locally uniformly for on $\R$}.
	\end{equation}
\end{theorem}

Theorem \ref{theo-ass-prop} thus provides sufficient conditions for the  existence of $L^{\rm prop}_\ep<+\infty$, see \eqref{DEF-Lprop}, in the bistable case, and thus in the ignition case due to comparison arguments.
  
For classical diffusion bistable equations, the fact that \lq\lq sufficiently large'' step initial data lead to propagation is well-known,  see \cite{Aro-Wei-78} for \lq\lq sufficiently high'' step initial data, and \cite{fife1977approach} for step initial data with any height larger than $\theta$. In the nonlocal diffusion setting, let us mention two propagation threshold properties proved by Berestycki and Rodr\' iguez \cite[Theorem 10]{berestycki2017non} and Lim
\cite[Theorem 1.2]{limlong}  for \lq\lq sufficiently high'' step initial data and under the assumption that the kernel is exponentially bounded in the sense of \eqref{expo-decay}. These results also require a structure assumption as 
\eqref{hyp:g-monotone-prop} or a \lq\lq $c>0$ traveling wave'' assumption. They 
mainly rely on the construction of a compactly supported sub-solution using energy methods \cite{berestycki2017non} or traveling wave solutions \cite{limlong}. Theorem \ref{theo-ass-prop}  improves these results into two directions.  First, we cover more dispersal kernels by only requiring the existence of a finite first moment, see \eqref{first_moment}.  Second, and even more important, we allow  the height of the step initial data which we consider to be $\theta+\ep$ for any small $\ep>0$. This is achieved by performing a nonlocal genralization of the arguments in \cite{fife1977approach}. This  requires a finite first moment of the dispersal kernel to  guarantee the existence of a traveling wave solution $(c,U)$ and to collect some useful integrability properties of the wave $U$ at $\pm \infty$. It would be interesting  to  investigate on the  existence (or not) of a propagation threshold for very heavy kernels not admitting a finite first moment, and for step initial data with height $\theta+\ep$. We believe that this is a delicate issue.

\section{Extinction criterion}\label{extinction}

In this section we consider the following semilinear problem
\begin{equation}\label{super}
\begin{cases}
\partial _t w=J\ast w-w+g(w), & t>0, \, x\in\R,\\
w(0, x)=(\theta+\ep)\1_{(-L, L)}(x), & x\in\R,
\end{cases}
\end{equation}
where $\ep>0$, $L>0$ and 
$$
g(u):=r^+(u-\theta)_+.
$$
Here $r^+>0$ and $\theta>0$ are given and fixed parameter and the subscript $+$ is used to denote the positive part of a real number. We start with a criterion for extinction which does not require $\ep$ to be small.

\begin{proposition}[Extinction criterion]\label{prop:extinction} Let Assumption \ref{ass:kernel} be satisfied. Let $\ep>0$ be fixed. Assume that $T>0$ and $L>0$ are such that
\begin{equation}\label{ext-criterion}
\theta e^{-(r^++1)T}\sum_{i=1}^{+\infty} \frac{T^i}{i!}R_i(L)\geq \ep.
\end{equation}
	Then the solution $w=w(t, x)$ to \eqref{super} satisfies
	\begin{eqnarray}
	\sup_{x\in\R}w(T, x)\leq\theta,
	\end{eqnarray}
	and is thus going to extinction at large times.
\end{proposition}

\begin{proof}
    Consider $v=v(t, x)$ the solution to the  linear Cauchy problem
	\begin{equation*}
	\begin{cases}
	\partial _t v=J\ast v-v, & t>0, \, x\in\R,\\
	v(0, x)=w(0, x), & x\in\R.
	\end{cases}
	\end{equation*}  
Recalling the definitions of $K$ and $\psi$ in \eqref{def-K} and \eqref{def-psi}, one gets $v(t, x)=(\theta+\ep)K(t, \cdot)\ast \1_{(-L, L)}(x)$. Let us first prove, through a regularization argument and \cite{xu2021spatial}, that
    \begin{equation}\label{def-V}
    V(t):=\|v(t, \cdot)\|_{L^\infty(\R)}=(\theta+\ep)\int_{|x|< L}K(t, x)dx=(\theta+\ep)\left[\int_{|x|< L}\psi(t, x)dx+e^{-t}\right].
    \end{equation}
Take two sequences $(u_n^-)_{n\geq 0}$, $(u_n^+)_{n\geq 0}$ of continuous, symmetric and nonincreasing on $\R^+$ functions such that, for all $n\geq 0$ and $x\in \R$, 
$$
0\leq u_n^-(x)\leq w(0,x)\leq u_n^+(x)\leq (\theta+\ep)\1_{(-2L, 2L)}(x),
$$
$u_n^-(0)=\theta+\ep$, and $u_n^\pm (x)\to w(0,x)$, as $n\to+\infty$, for almost all $x\in\R$. Denoting  by $v_n^\pm=v_n^\pm(t,x)$ the solutions to $\partial _t v=J*v-v$ starting from $u_n^\pm$, the comparison principle yields
\begin{equation}\label{cp-yields}
0\leq v_n^-(t,x)\leq v(t,x)\leq v_n^+(t,x)\leq \theta+\ep.
\end{equation}
Now, since the initial data $u_n^\pm$ are continuous, symmetric and decreasing on $\R^+$ and since the kernel $J$ is symmetric, Xu et al. \cite[Theorem 2.5]{xu2021spatial} proved that, for each $t>0$ and $n\geq 0$, both functions $x\mapsto v_n^-(t,x)$ and $x\mapsto v_n^+(t,x)$ are also symmetric and nonincreasing on $\R^+$. Hence one obtains, for $t>0$ and $n\geq 0$ , 
	$$
	\|v_n^\pm(t, \cdot)\|_{L^\infty(\R)}=v_n^\pm(t,0)=\int_{\R}\psi(t, x)u_n^\pm(x) dx+e^{-t}u_n^\pm(0).
	$$
From \eqref{cp-yields} we deduce that, for any $t>0$ and $n\geq 0$,
$$
\int_{\R}\psi(t, x)u_n^-(x) dx+e^{-t}u_n^-(0)\leq V(t)\leq 
\int_{\R}\psi(t, x)u_n^+(x) dx+e^{-t}u_n^+(0).
$$	
Recalling that $u_n^\pm(0)=\theta+\ep$, we let $n\to+\infty$ which, using Lebesgue convergence theorem, yields \eqref{def-V}.

    Following  Alfaro et al. \cite{alfaro2020quantitative}, we aim at constructing a super-solution to \eqref{super} in the form  $W(t, x):=v(t, x)\varphi(t)$, with $\varphi(0)=1$ and $\varphi(t)>0$. As in  \cite[Proof of Proposition 2.1]{alfaro2020quantitative}, we choose 
	$$ \varphi(t)=e^{r^+t}\left(1-\int_{0}^{t}r^+e^{-r^+s}\frac{\theta}{V(s)}ds\right). $$
	Observe that $V(0)\varphi(0)>\theta$ and denote by $T>0$ the first time where $V(T)\varphi(T)=\theta$ (obviously we let $T=+\infty$ if such a time does not exist). Then $\left(\varphi(t)-\frac{\theta}{V(t)}\right)_+= \varphi(t)-\frac{\theta}{V(t)}$ for all $t\in [0,T)$, and thus one can check that $W(t,x)=v(t,x)\varphi(t)$ is a super-solution to \eqref{super} on the time interval $(0,T)$. In particular, if $T<+\infty$, it follows from the comparison principle that $w(T,\cdot)\leq W(T,\cdot)\leq \theta$, and we are done. 

    The condition $T<+\infty$ rewrites as: there exists $T>0$ such that $F_L(T)=0$ where
    $$
    F_L(t):=\theta\left(1-\frac{e^{-r^+t}}{A_L(t)}\right)+\ep-\int_{0}^{t}r^+e^{-r^+s}\frac{\theta}{A_L(s)}ds=0,
    $$
	wherein
	$$ 
    A_L(t):=\int_{|x|< L}\psi (t, x)dx+e^{-t}. 
	$$
    We claim (see below for a proof) that $A_L'(t)\leq 0$ for all $t>0$. As a result, since $F_L(0)=\ep$ and $F'_L(t)=\theta e^{-r^+t}\frac{A'_L(t)}{A^2_L(t)}\leq0$, we are left to find $T>0$ such that $F_L(T)\leq 0$, that is
	\begin{equation*}
	1+\frac{\ep}{\theta}\leq r^+\int_{0}^{T}\frac{e^{-r^+s}}{A_L(s)}ds+\frac{e^{-r^+T}}{A_L(T)}.
	\end{equation*}
    Integrating by parts this is equivalent to 
    \begin{equation}\label{to-be-reached}
    \frac{\ep}{\theta}\leq -\int_{0}^{T}e^{-r^+s}\frac{A'_L(s)}{A_L^{2}(s)}ds.
    \end{equation}
    But, since $0<A_L(s)\leq1$ (recall  \eqref{1-e^t}) and $A'_L(s)\leq 0$ for all $s>0$, 
    \begin{eqnarray}
    -\int_{0}^{T}e^{-r^+s}\frac{A'_L(s)}{A_L^{2}(s)}ds
    &\geq & -\int_{0}^{T}e^{-r^+s}A'_L(s)ds\nonumber\\
    &\geq & e^{-r^+T}(1-A_L(T))\nonumber\\
    & = & e^{-r^+T}\left(1-e^{-T}-\int_{|x|< L}\psi (T, x)dx\right)\nonumber\\
    &= & e^{-r^+T}\left(1-e^{-T}-e^{-T}\sum_{i=1}^{+\infty} \frac{T^i}{i!}\int_{|x|< L}J^{*(i)}(x)dx\right)\nonumber\\
    &= & e^{-r^+T}\left(1-e^{-T}-e^{-T}\sum_{i=1}^{+\infty} \frac{T^i}{i!}\left(1-R_i(L)\right)\right)\nonumber\\
    &=& e^{-r^+T}\left(1-e^{-T}-e^{-T}(e^T-1)+e^{-T}\sum_{i=1}^{+\infty} \frac{T^i}{i!}R_i(L)\right)\nonumber\\
    &=& e^{-(r^++1)T}\sum_{i=1}^{+\infty} \frac{T^i}{i!}R_i(L).\nonumber
    \end{eqnarray}
    From this and \eqref{to-be-reached}, we conclude that \eqref{ext-criterion} enforces $T<+\infty$ and is thus a sufficient condition for extinction. 
\end{proof}

\medskip

It remains to prove the claim that $A'_L(t)\leq0$ for all $t>0$.

\medskip

\begin{proof}
	Observe that $A_L(t)=e^{-t}+\int_{\R}\psi(t, x)u_0(x)dx$, where $u_0=\1_{(-L, L)}$, and thus 
	\begin{eqnarray}
    A'_L(t)
	&=&-e^{-t}+\int_{\R}\partial_t\psi(t, x)u_0(x)dx\nonumber\\
	&=&-e^{-t}u_0(0)+\int_{\R}\left(J\ast \psi(t, \cdot)(x)-\psi(t,x)+e^{-t}J(x)\right)u_0(x)dx\nonumber\\
	&=&\int_{\R}J(y)\int_{\R}\left(\psi(t,x-y)u_0(x)-\psi(t, x)u_0(x)\right)dxdy+e^{-t}\int_{\R}J(x)\left(u_0(x)-u_0(0)\right)dx\nonumber\\
	&=&\int_{\R}J(y)\left(\psi(t,\cdot)*u_0(y)-\psi(t, \cdot)*u_0(0)\right)dy+e^{-t}\int_{\R}J(x)\left(u_0(x)-u_0(0)\right)dx\nonumber\\
	&=&\int _{\R} J(x)\left(u(t,x)-u(t,0)\right)dx,\nonumber
	\end{eqnarray}
	where we have used \eqref{partial} and the symmetry of $\psi$ and where $u=u(t,x)$ is the solution of $\partial _t u=J*u-u$ starting from $u_0$.  Now, since  both $J$ and $u_0$ are symmetric and nonincreasing on $\R^+$, it follows from Xu et al. \cite[Theorem 2.5] {xu2021spatial} (and a regularization argument as above) that $u(t,x)\leq u(t,0)$ for all $t>0$, $x\in \R$. Thus we have $A'_L(t)\leq0$ for all $t>0$.
\end{proof}

\medskip

We now take a kernel $J$ satisfying Assumptions \ref{ass:kernel} and \ref{ass:kernel-bis} with, in particular, the expansion \eqref{J-Fourier-ass} for some $0<\be\leq 2$. We consider the above extinction criterion in the limit $\ep \to 0$, revealing the role of the dispersal kernel.  Notice that the results below connect to the {\it control issue}: given a \lq\lq final time'' $T>0$, how large can we choose the size of the initial data so that the solution is everywhere smaller than the threshold $\theta$ at time $T$?

We start with the case $0<\be< 2$. Our analysis seems to reveal that the \lq\lq final time'' does not affect the order of magnitude  $L\sim  \ep^{-1/\be}$ but plays a role in the value of the \lq\lq multiplicative constant''.

\begin{corollary}[Asymptotic extinction criterion, $0<\be< 2$]\label{asmptotic}
Let Assumptions \ref{ass:kernel} and \ref{ass:kernel-bis} be satisfied with $0<\be< 2$. Let $T>0$ be a fixed time. Then, for any $0<\ga<1$, there exists $\ep_0>0$ small enough such that, for each $\ep\in(0, \ep_0)$ and for each 
  $$
  0<L<\ga \left(\theta aC Te^{-r^+T}\right)^{1/\be}\ep ^{-1/\be},
  $$
the solution $w=w(t, x)$ to \eqref{super} is everywhere smaller than $\theta$ at time $T$ and is thus going to extinction at large times. Here constants $a>0$ and $C=C(\beta)>0$ come from the asymptotic expansion \eqref{J-Fourier-ass} and Lemma \ref{lem:tails}-(ii), respectively.
\end{corollary}

\begin{proof}  
	For a given $0<\ga<1$, we select an integer $N$ large enough so that 
	$$
	\sum_{i=0}^N \frac{T^{i}}{i!}\geq \ga ^{\be/2}e^T. 
	$$
	Next, from \eqref{tails-beta-lower-2}, there is $L_0>0$ large enough such that, for all $L\geq L_0$ and $1\leq i \leq N+1$, 
	$$
	R_i(L)\geq \ga ^{\be/2} i\, \frac{aC}{L^\be}. 
	$$
	Hence, if $L\geq L_0$ then
    $$
    \theta e^{-(r^++1)T}\sum_{i=1}^{+\infty} \frac{T^i}{i!}R_i(L)\geq \theta e^{-(r^++1)T}\sum_{i=1}^{N+1}\frac{T^i}{i!}\ga^{\be/2} i\,\frac{aC}{L^\be}\geq \ga ^{\be} \theta T e^{-r^+T}\frac{aC}{L^\be}.
    $$
    As a result, to check the extinction criterion \eqref{ext-criterion} it is enough to have 
    $$
    L\geq L_0 \quad \text{ and }\quad\ep \leq \ga ^{\be} \theta T e^{-r^+T}\frac{aC}{L^\be}.
    $$
    Hence, there is $\ep _0>0$ such that, for all $\ep \in (0,\ep_0)$, the conclusion of the corollary holds true for $L=\ga \left(\theta aC Te^{-r^+T}\right)^{1/\be}\ep ^{-1/\be}$, and also for smaller $L$ from the comparison principle.
\end{proof}

\medskip

On the other hand, the case $\be=2$ is more tricky as revealed by the following corollaries.

\begin{corollary}[Asymptotic extinction criterion for exponentially bounded kernels]\label{lighttail}
	Let Assumption \ref{ass:kernel} be satisfied and assume further the exponential decay \eqref{expo-decay}. Then there exists $\ep_0>0$ small enough such that, for each $\ep\in(0, \ep_0)$ and for each $L$  satisfying  
	$$
	0<L<C^-\ln\frac{1}{\ep},
	$$
	with $C^-=C^-(r^+,J)>0$,  the solution $w=w(t, x)$ to \eqref{super} is everywhere smaller than $\theta$ at time $T=\tilde C^{-}\ln \frac 1 \ep$, $\tilde C^-=\tilde C^-(r^+,J)>0$, and is thus going to extinction at large times.
\end{corollary}

\begin{proof} We shall use the large deviations Cram\'er theorem, as stated in  Lemma \ref{lem:expo}. We use the criterion \eqref{ext-criterion} with a single  \lq\lq large'' term $i=i(\ep)$: to prove extinction  it is sufficient to obtain $
\theta e^{-(r^++1)T}\frac{T^i}{i!} R_i(L) \ge \ep$ or, equivalently, 
\begin{equation}\label{exp-super}
-(r^++1)\frac T i+\ln T-\frac 1 i \ln (i!)+\frac 1 i \ln R_{i}(L)\geq \frac 1 i (\ln\epsilon-\ln\theta).
\end{equation}
Recalling that $L_1>0$ is defined in Lemma \ref{lem:expo} let us fix $0<L_0<s_0<L_1$. Then since $\La^*$ is nondecreasing on $(0,+\infty)$, one has
 $0\leq \Lambda^*(L_0)\leq \La^*(s_0)\leq \Lambda^*(L_1)<+\infty$, while according to Lemma \ref{lem:expo} the convergence in \eqref{cramer} holds uniformly in any compact subset of $(0,L_1]$. 
 We now choose 
$$L=Cs_0\ln\frac{1}{\epsilon}, \quad T=i=[C\ln \frac 1 \ep],
$$
where $C>0$ is some constant to be determined so that \eqref{exp-super} is satisfied for all $\ep$ small enough.

From Stirling's formula, we see that there is $K>0$ such that $\frac 1 i \ln (i!)\leq \ln i +K$ for all $i\geq 1$. Hence the condition \eqref{exp-super} is reached as soon as
\begin{equation*}
-(r^++1)-K+\frac 1 i \ln R_{i}(L)\geq \frac 1 i(\ln\epsilon-\ln\theta).
\end{equation*}
As $\ep\to 0$, the right hand side tends to $-1/C$. On the other hand, we write
\begin{equation*}
\frac 1 i\ln R_{i}(L)=\left(\frac 1 i \ln R_i\left(i\frac L i\right)+\Lambda^*\left(\frac L i\right)\right)-\La^*\left(\frac{L}{i}\right).
\end{equation*}
Since $s_0\leq \frac L i\leq \frac{Cs_0\ln \frac 1 \ep}{C\ln \frac 1 \ep -1}$,
the first term in the right-hand converges to $0$ as $\ep\to 0$ due to the uniformity of the convergence in \eqref{cramer}, while the monotonicity of $\La^*$ yields, for all $\ep$ small enough,
$$
-\La^*\left(\frac{L}{i}\right)\geq -\La^*(L_1).
$$ 
As a consequence we get
\begin{equation*}
\liminf_{\ep\to 0}\frac 1 i\ln R_{i}(L)\geq -\La^*(L_1).
\end{equation*}
As a result it suffices to have $-(r^++1)-K-\Lambda^*(L_1)\geq -1/C$, which is achieved by taking $C=C(r^+,s_0)=C(r^+,J)>0$ small enough. The extinction for smaller $L$ follows from the comparison principle.
\end{proof}

\begin{corollary}[Asymptotic extinction criterion, some examples with $\be=2$]\label{remarkbeta=2}
	Let Assumption \ref{ass:kernel} be satisfied.  Let $T>0$ be a fixed time. 
	\begin{itemize}
	\item [(i)] Assume that there are $c_2>c_1>0$, $\alpha>2$ and $x_0\in \R$ such that $ \frac{c_1}{\vert x \vert^{1+\alpha}}\leq J(x) \leq  \frac{c_2}{\vert x \vert^{1+\alpha}}$ for all $\vert x\vert \geq \vert x_0\vert $. Then there is $C^-=C^-(\theta,r^+,T,c_1,\alpha)>0$ such that, for each $\ep>0$  and each $L$ satisfying
	$$
	0<L<C^-\frac 1 {\ep^{\frac{1}{\al}}},
	$$
    the solution $w=w(t, x)$ to \eqref{super} is everywhere smaller than $\theta$ at time $T$ and is thus going to extinction at large times.
    
    \item[(ii)]	 Assume that $J$ has regularly varying tails $RV(\al)$ with $\alpha>2$. Let us fix $\alpha'>\alpha$. Then there exists $\ep_0>0$ small enough such that, for each $\ep\in(0, \ep_0)$ and for each $L$ satisfying  
	$$
	0<L<\frac 1 {\ep^{\frac{1}{\al'}}} ,
	$$
    the solution $w=w(t, x)$ to \eqref{super} is everywhere smaller than $\theta$ at time $T$ and is thus going to extinction at large times.	
    
    \item [(iii)]  Assume that $J$ has lognormal-type tails $LN(\gamma,\lambda,\rho)$ with $\gamma>1$, $\lambda>0$ and $\rho \in \R$. Let us fix $\la '>\la$. Then there exists $\ep_0>0$ small enough such that, for each $\ep\in(0, \ep_0)$ and for each $L$ satisfying 
	$$
    0<L< e^{\left(\frac 1 {\la '} \ln\frac 1 \ep\right)^{\frac 1 \ga}},
	$$
    the solution $w=w(t, x)$ to \eqref{super} is everywhere smaller than $\theta$ at time $T$ and is thus going to extinction at large times. 	  
     
    \item [(iv)]  Assume that $J$ has  Weibull-like tails $WE(\alpha,\lambda,\rho)$ with $0<\al <1$, $\lambda>0$ and $\rho \in \R$. Let us fix $\la '>\la$. Then there exists $\ep_0>0$ small enough such that, for each $\ep\in(0, \ep_0)$ and for each $L$ satisfying  
    \begin{equation}\label{bidule}
    0< L < \left(\frac 1 {\la '} \ln \frac 1 \ep\right)^{\frac 1 \al},
    \end{equation}
    the solution $w=w(t, x)$ to \eqref{super} is everywhere smaller than $\theta$ at time $T$ and is thus going to extinction at large times.	   
\end{itemize}
\end{corollary}

\begin{proof} We use the criterion \eqref{ext-criterion} with only the term $i=1$: to prove extinction  it is sufficient to obtain $\theta e^{-(r^++1)T}T R_1(L)\geq  \ep$. 

In case $(i)$ since $R_1(L)\geq \frac{2c_1}{\alpha}L^{-\alpha}$ the conclusion easily follows.

In case $(ii)$, from \eqref{tail-RV} the condition is recast
$$
CS(L)L^{-\alpha}\geq \ep, \quad C:=\theta e^{-(r^++1)T}T,
$$
where $S$ is a slowly varying function. Letting $L=\frac 1 {\ep^{\frac{1}{\al'}}}$, this is recast 
$$
CL^{\alpha'-\alpha}S(L)\geq 1,
$$
which is true for $L\gg 1$ from known properties of slowly varying functions, see e.g. \cite[VIII, Lemma 2]{Feller}. As a result, extinction does occur for $0<\ep<\ep _0$ with $\ep_0>0$ small enough and $L=\frac 1 {\ep^{\frac{1}{\al'}}}$. The extinction for smaller $L$ follows from the comparison principle. 
	
In case $(iii)$, from \eqref{tail-LN} the condition is recast
$$
C L^\rho e^{-\lambda \ln ^\gamma L}\geq \ep, \quad C:=\theta e^{-(r^++1)T}T c.
$$
Letting $L=e^{\left(\frac 1 {\la '} \ln\frac 1 \ep\right)^{\frac 1 \ga}}$, this is recast 
$$
CL^{\rho}e^{(\la '-\la) \ln ^\gamma L}\geq 1,
$$
which is true for $L\gg 1$. As a result, extinction does occur for $0<\ep<\ep _0$ with $\ep_0>0$ small enough and $L=e^{\left(\frac 1 {\la '} \ln\frac 1 \ep\right)^{\frac 1 \ga}}$. The extinction for smaller $L$ follows from the comparison principle. 

In case $(iv)$, we recast the condition thanks to \eqref{tail-WE} and use a similar argument.
\end{proof}

\begin{remark}\label{rem:as-easily}  As easily seen from the proof,  the estimate in case $(ii)$ can be slightly improved (namely $\alpha '$ can be taken equal to $\alpha$) when the slowly varying function $S$ satisfies $S(L)\to +\infty$ as $L\to +\infty$. Similarly the estimates in $(iii)$ and $(iv)$ can be slightly improved (namely $\la '$ can be taken equal to $\la$) when $\rho >0$. Last, in $(iv)$, when $\rho=0$, \eqref{bidule} can be replaced by
\begin{equation}\label{bidule2}
0<L<\left(\frac{1}{\lambda}\ln \frac C \ep\right)^{\frac 1 \alpha},\text{ where }C:=\theta e^{-(r^++1)T}Tc.
\end{equation}
\end{remark}

\section{Non-extinction criterion}\label{non-extinction}

In this section, we fix $r^->0$, $\theta\in(0, 1)$, and define the linear function 
\begin{equation}\label{g}
\tilde{g}(w):=r^-(w-\theta).
\end{equation} 
For $\ep>0$ and $L>0$, we consider $w_L^\ep=w_L^\ep(t, x)$ the solution of the Cauchy problem
\begin{equation}\label{sub}
\begin{cases}
\partial _t w=J\ast w-w+\tilde{g}(w), &  t>0,\, x\in\R,\\
w(0, x)=(\theta+\ep)\1_{(-L, L)}(x), &  x\in\R.\end{cases}
\end{equation}
We start with a criterion for non-extinction which does not require $\ep$ to be small.

\begin{proposition}[Non-extinction criterion]\label{nonextinction}
Let Assumption \ref{ass:kernel}-(i) be satisfied. Let $\ep>0$ and $L>0$ be given. Let $\eta\in(\theta, 1)$ and $m\in (0,1)$ be given.  Define
\begin{equation}\label{def:Tep}
T_\ep:=\frac{1}{r^-}\ln \frac{2(\eta-\theta)}{\ep}.
\end{equation}

\begin{itemize}
\item [(i)] For all $0<t\leq T_\ep$, all $x\in \R$, $w_L^\ep(t, x)\leq \theta+2(\eta-\theta)$.

\item [(ii)] If 
\begin{equation}\label{nonextinctioncondition}
\frac{\ep}{2(\theta+\ep)}\geq e^{-T_\ep }\sum_{i=1}^{+\infty} \frac{T_\ep^i}{i!}R_i\left((1-m)L)\right),
\end{equation}
then
\begin{equation}\label{conclusion}
\min_{|x|\leq mL} w_L^\ep(T_\ep, x)\geq \eta.
\end{equation}
\end{itemize}
\end{proposition}

\begin{proof}
	Notice that the function $w^\ep_L(t, x)$ is given by 
	\begin{equation}\label{w^ep_L}
	w^\ep_L(t, x)=\theta+e^{r^-t}(v(t, x)-\theta), 
	\end{equation}
	where $v=v(t, x)$ denotes the solution of the linear equation
	$$ 
	\partial _t v =J\ast v-v, \; t>0, \, x\in\R, 
	$$
starting from $w^\ep_L(0, x)=(\theta+\ep)\1_{(-L, L)}(x)$. From the comparison principle $v(t,x)\leq \theta +\ep$ for all $t>0$ and $x\in \R$, and thus, for all $0<t\leq T_\ep$,
	$$
	w^\ep(t,x)\leq \theta+\frac{2(\eta-\theta)}{\ep}(\theta+\ep-\theta)=\theta+2(\eta-\theta),\; \forall x\in\R,
	$$ 
	which proves $(i)$.

	Next, we know from subsection \ref{ss:linear} that $v=v(t,x)$ is given by
	$$
	v(t,x)=(\theta+\ep) \left[e^{-t} \1_{(-L, L)}(x)+\int_{-L}^L\psi(t,x-y)dy\right].
	$$	
	From \eqref{1-e^t} we obtain	
	$$
	\frac {v(t,x)}{\theta+\ep}=1+e^{-t} \left(\1_{(-L, L)}(x)-1\right)-\int_{|y|\geq L}\psi(t,x-y)dy,
	$$
and thus
$$
	w^\ep_L(t,x)=\theta+(\theta+\ep)e^{r^-t}\left[e^{-t} \left(\1_{(-L, L)}(x)-1\right)-\int_{|y|\geq L}\psi(t,x-y)dy+\frac{\ep}{\theta+\ep}\right].
$$	
We now restrict to $x$ satisfying $|x|\leq mL$.  Since $|y|\geq L$ ensures that $|x|\leq mL\leq m|y|$ and $|x-y|\geq(1-m)|y|$, we deduce that
	\begin{eqnarray}
	w^\ep_L(t, x)
	&\geq& \theta+(\theta+\ep)e^{r^-t}\left[\frac{\ep}{\theta+\ep}-\int_{|z|\geq (1-m)L}\psi(t, z)dz\right]\nonumber\\
	&=& \theta+(\theta+\ep)e^{r^-t}\left[\frac{\ep}{\theta+\ep}-e^{-t}\sum_{i=1}^{+\infty} \frac{t^i}{i!}\int_{\vert z\vert \geq (1-m)L}J^{*(i)}(z)dz\right].\nonumber
	\end{eqnarray}
It follows from this and \eqref{nonextinctioncondition} that
$$
	\min_{|x|\leq mL}w^\ep_L(T_\ep, x)\geq \theta+(\theta+\ep)\frac{2(\eta-\theta)}{\ep}\left[\frac{\ep}{\theta+\ep}-\frac{\ep}{2(\theta+\ep)}\right]=\eta,
$$
which concludes the proof of $(ii)$.
\end{proof}

\begin{remark}\label{Rem1}
Note that the right hand side of \eqref{nonextinctioncondition} is decreasing with respect to $L>0$ from $1-e^{-T_\ep}$ when $L=0$ to $0$ when $L\to +\infty$. 
\end{remark}

We now take a kernel $J$ satisfying Assumptions \ref{ass:kernel}-$(i)$ and \ref{ass:kernel-bis} with, in particular, the expansion \eqref{J-Fourier-ass} for some $0<\be\leq 2$. We consider the above non-extinction criterion in the limit $\ep \to 0$, revealing the role of the dispersal kernel. 

\begin{corollary}[Asymptotic non-extinction criterion, $0<\be\le2$]\label{prop:0<beta<2}
Let Assumptions \ref{ass:kernel}-(i) and \ref{ass:kernel-bis} with $0<\be\le2$ be satisfied. Let $\eta\in(\theta, 1)$ and $m\in(0, 1)$ be given. Then for all $\ep>0$, there is $L_\ep>0$ such that the conclusion \eqref{conclusion} holds for all $L\geq L_\ep$. Furthermore, there is $\ep_0>0$ small enough such that, for all $0<\ep<\ep_0$, $L_\ep$ can be chosen as follows,
$$
L_\ep=\frac{C^+}{1-m}\left(\frac{1}{\ep}\ln\frac{1}{\ep}\right)^{\frac1\be},
$$
for some constant $C^+=C^+(\theta, r^-, J)>0$.
\end{corollary}
\begin{proof}
	Observe first that the existence of $L_\ep$ directly follows from Remark \ref{Rem1}. We now turn to the asymptotic of $L_\ep$ as $\ep\to 0$. From Lemma \ref{lem:Durret} we can write
	\begin{equation*}
	\begin{split}
	e^{-T_\ep}\sum_{i=1}^{\infty} \frac{T_\ep^i}{i!} R_i\left((1-m)L\right)&\leq e^{-T_\ep}\sum_{i=1}^{\infty} \frac{T_\ep^i}{i!}\frac{(1-m)}{2}L \int_{|\xi|\leq \frac{2}{(1-m)L}}\left[1-\left(\widehat J(\xi)\right)^i\right]d\xi\\
	&\leq e^{-T_\ep}\sum_{i=1}^{\infty} \frac{T_\ep^i}{i!}\frac{(1-m)}{2}L \int_{|\xi|\leq \frac{2}{(1-m)L}}i (1-\widehat J(\xi))d\xi,
	\end{split}
	\end{equation*}
	from the mean value theorem. Hence, from \eqref{J-Fourier-ass}, there is  $C=C(\be,a)>0$ such that, for all $\ep>0$ and all $L>0$ large enough,
	\begin{equation*}
	\begin{split}
	e^{-T_\ep}\sum_{i=1}^{\infty} \frac{T_\ep^i}{i!} R_i\left((1-m)L\right)
	&\leq C e^{-T_\ep}\sum_{i=1}^{\infty} \frac{iT_\ep^i}{i!}\frac{1}{(1-m)^\be L^\be}=C \frac{T_\ep}{(1-m)^\be L^\be}.
	\end{split}
	\end{equation*}
	From this and the definition of $T_\ep$ in \eqref{def:Tep}, we see that the non-extinction criterion \eqref{nonextinctioncondition} is satisfied, for all $\ep>0$ small enough, as soon as $L>L_\ep:=\frac{C^+}{1-m}(\frac 1 \ep \ln \frac 1 \ep)^{\frac1\beta}$ for some $C^+=C^+(\theta,r^-,\beta,a)>0$, which proves the desired result.
\end{proof}

\medskip

On the other hand, the case $\be=2$ is more tricky as revealed by the following corollaries. Recall that $r^-$ is defined in \eqref{r-moins} and $m_2(J)=\int_{\R}x^2J(x)dx$.

\begin{corollary}[Asymptotic non-extinction criterion for exponentially bounded kernels]\label{prop:exp-decay}
	Let Assumption \ref{ass:kernel}-(i) be satisfied. Assume further the exponential decay \eqref{expo-decay} and $r^-\in(0, 1)$. Let $\eta\in(\theta, 1)$ and $m\in(0, 1)$ be given. Then for all $\ep>0$, there is $L_\ep>0$ such that the conclusion \eqref{conclusion} holds for all $L\geq L_\ep$. Furthermore, there is $\ep_0>0$ small enough such that, for all $0<\ep<\ep_0$, $L_\ep$ can be chosen as follows,
	$$
L_\ep=\frac{C^+}{1-m}\ln\frac{1}{\ep},
	$$
	for some constant $C^+=C^+(r^-, J)>0.
	$
\end{corollary}

\begin{corollary}[Asymptotic non-extinction criterion, some examples with $\be=2$]\label{prop:examples-be=2}
	Let Assumption \ref{ass:kernel}-(i) be satisfied. Assume $r^-\in(0, 1)$ and $m_2(J)=1$. Let $\eta\in(\theta, 1)$ and $m\in(0, 1)$ be given. Then for all $\ep>0$, there is $L_\ep>0$ such that the conclusion \eqref{conclusion} holds for all $L\geq L_\ep$. Furthermore, there is $\ep_0>0$ small enough (possibly depending on $\tilde \alpha$ or $\tilde \lambda$ in cases $(ii), (iii)$ and $(iv)$, see below) such that, for all $0<\ep<\ep_0$, $L_\ep$ can be chosen as follows.
    \begin{itemize}
	
	\item [(i)] Assume that there are $c_2>c_1>0$, $\alpha>2$ and $x_0\in \R$ such that $\frac{c_1}{\vert x\vert^{1+\alpha}}\leq J(x)\leq \frac{c_2}{\vert x\vert ^{1+\alpha}}$ for all $\vert x\vert  \geq \vert x_0\vert$, then 
	$$
	L_\ep= \frac{C^+}{1-m}\left(\frac1\ep\ln\frac1\ep\right)^{\frac1\al},  
	$$
	for some constant $C^+=C^+(\theta, r^-, c_2, \alpha)>0$.
	
	\item [(ii)] Assume that $J$ has regularly varying tails $RV(\al)$ with $\al>2$, then 
	$$
	L_\ep= \left(\frac1\ep\ln \frac1\ep\right)^{\frac{1}{\tilde \al}}, 
	$$
	for any given $0<\tilde \alpha<\alpha$. 
	
	\item[(iii)] Assume that $J$ has lognormal-type tails $LN(\ga,\la,\rho)$ with $\ga>1$, $\la>0$, $\rho \in \R$, then 
	$$
	L_\ep=  \exp \left[\left(\frac{1}{\tilde \lambda} \ln \frac 1\ep\right)^\frac{1}{\ga}\right],
	$$  
	for any given $0<\tilde \lambda<\lambda$.
	
	\item[(iv)] Assume that $J$ has Weibull-like tails $WE(\al,\la,\rho)$ with $0<\al<1$, $\la>0$, $\rho \in \R$ then 
	$$
	L_\ep= \frac{1}{1-m}\left(\frac{1}{\tilde \lambda}\ln \frac 1\ep\right)^{\frac 1\alpha}, \quad \text{ if } \quad 0<\al<\frac 2 3,
	$$
    for any given $0<\tilde \lambda<\lambda$, while	
	$$
	L_\ep =\left(\ln \frac 1 \ep \right)^{\frac{1}{2(1-\tilde \alpha)}},\quad  \text{ if } \quad \frac 23\leq \al<1,
    $$
    for any given $0<\alpha<\tilde \alpha<1$. 
\end{itemize}
\end{corollary}
\begin{corollary}[Asymptotic non-extinction criterion, Weibull-like tails $WE(\al, \la, \rho=0)$]\label{prop:Weibull}
	Let Assumption \ref{ass:kernel}-(i) be satisfied. Further assume that $J$ has Weibull-like tails $WE(\al, \la, \rho=0)$ and $r^-\in(0, 1)$. Let $\eta\in(\theta, 1)$ and $m\in(0, 1)$ be given. Then for all $\ep>0$, there is $L_\ep>0$ such that the conclusion \eqref{conclusion} holds for all $L\geq L_\ep$. Furthermore, there is $\ep_0>0$ small enough such that, for all $0<\ep<\ep_0$, $L_\ep$ can be chosen as follows,
	$$
	L_\ep=\frac{C^+}{1-m}\left(\frac 1 \lambda \ln \frac 1 \ep\right)^{\frac 1\alpha},
	$$
	for some $C^+=C^+(J)>0$.
\end{corollary}
	
Again the existence of $L_\ep$ in Corollaries \ref{prop:exp-decay}, \ref{prop:examples-be=2} and \ref{prop:Weibull} directly follows from Remark \ref{Rem1}. We now turn to the estimates in Corollaries \ref{prop:exp-decay}, \ref{prop:examples-be=2} and \ref{prop:Weibull}. We assume 
\begin{equation}\label{hypr-}
0<r^-<1.
\end{equation}
As a preparation we consider the sum of the series in \eqref{nonextinctioncondition} and roughly  show that, for $t$ large enough,
the main contribution  corresponds to the indexes $i$ around $t$. To see this, observe that for any $t>0$ and any $M(t)\in\mathbb N$ such that $M(t)\leq t$ one has, using $i!\geq i^i e^{-i}$ and since $i\mapsto i+i\ln t-i\ln i$ is increasing on $(0,t)$, 
\begin{equation}
\label{M}
e^{-t}\sum_{i=1}^{M(t)} \frac{t^i}{i!}\leq e^{-t}\sum_{i=1}^{M(t)} \exp\left(i+i\ln t-i\ln i\right)\leq e^{-t} M(t)\left(\frac{et}{M(t)}\right)^{M(t)}.
\end{equation}
On the other hand, for any $t>0$ and any $N(t)\in \mathbb{N}$ such that $N(t)>et$, one has
\begin{equation}
\label{N}
e^{-t}\sum_{i=N(t)}^{+\infty} \frac{t^i}{i!}\le e^{-t} \sum_{i=N(t)}^{+\infty} \left(\frac{et}{i}\right)^i \le e^{-t}\sum_{i=N(t)}^{+\infty} \left(\frac{et}{N(t)}\right)^i= e^{-t}\left(\frac{et}{N(t)}\right)^{N(t)}\frac{N(t)}{N(t)-et}.
\end{equation}
Let us choose $0<\ga^-<1$ such that 
$$
\ga^-\ln (e/\ga^-)<1-r^-,
$$
and $M(t)=[\ga^- t]$.  Then $\gamma^-t-1<M(t)\leq \gamma ^- t < t$ and thus \eqref{M} yields
$$
e^{-t}\sum_{i=1}^{M(t)} \frac{t^i}{i!}\leq \ga^- t e^{-t} \exp\left(\ga^- t\ln \frac{et}{\ga^- t-1}\right),
$$
so that 
$$
e^{-t}\sum_{i=1}^{M(t)} \frac{t^i}{i!}\leq \ga^- t \exp\left[(\ga^- \ln (e/\ga^-)-1)t+\mathcal O(1)\right]=o(e^{-r^- t}), \quad \text { as } t \to +\infty.
$$
Let us next choose $N(t)=[3t]+1>et$, so that \eqref{N} yields
$$
e^{-t}\sum_{i=N(t)}^{\infty} \frac{t^i}{i!}\leq e^{-t}\frac{[3t]+1}{[3t]+1-et}=o(e^{-r^- t}), \quad \text { as } t \to +\infty.
$$
To summarize the above analysis,  since $J^{*(i)}\geq 0$ and $\int_{\R} J^{*(i)}(x)dx=1$ for any $i\geq 1$, we have proved the following lemma.
\begin{lemma}\label{LE-E1}
Assume that $0<r^-<1$. Select $\gamma^->0$ small enough so that $\ga^-\ln (e/\ga^-)<1-r^-$. Defining 
$$
M(t):=[\gamma^-t], \quad N(t):=[3t]+1,
$$
we have
\begin{equation}\label{E1}
e^{-t}\sum_{i=1}^{+\infty} \frac{t^i}{i!} R_i\left((1-m)L\right)=e^{-t}\sum_{i=M(t)}^{N(t)} \frac{t^i}{i!}R_i\left((1-m)L\right)+o\left(e^{-r^- t}\right), \quad \text { as } t \to +\infty,
\end{equation}
independently of $m\in(0,1)$ and $L>0$.
\end{lemma}

Let us recall that $T_\ep=\frac{1}{r^-}\ln \frac{2(\eta-\theta)}{\ep}$ and define $M_\ep:=M(T_\ep)$ and $N_\ep:=N(T_\ep)$. Then, from Lemma \ref{LE-E1}, we see that the non-extinction criterion \eqref{nonextinctioncondition} is satisfied as soon as 
\begin{equation}\label{estrb2}
\frac{\ep}{4\theta}\geq e^{-T_\ep}\sum_{i=M_\ep}^{N_\ep} \frac{T_\ep^i}{i!} R_i\left((1-m)L\right).
\end{equation}
holds for $\ep>0$ small enough. We now explore the new criterion \eqref{estrb2} with various assumptions.

\medskip

\begin{proof}[\bf Proof of Corollary \ref{prop:exp-decay}] Here we assume that $J$ is exponentially bounded in the sense of \eqref{expo-decay}.  Let us fix $\lambda>0$  so that $\La(\lambda)<+\infty$. Now choose $L=T_\ep\ell$ with $\ell>0$ to be determined later. It follows from \eqref{esti-expo} that
\begin{equation*}
\begin{split}
e^{-T_\ep}\sum_{i=M_\ep}^{N_\ep} \frac{T_\ep^i}{i!} R_i\left((1-m)L\right)&\leq e^{-T_\ep}\sum_{i=M_\ep}^{N_\ep} \frac{T_\ep^i}{i!} \exp\left(-i\left(\lambda\frac{(1-m)\ell T_\ep}{i}-\La(\lambda)\right)\right)\\
&\leq e^{-T_\ep}\sum_{i=M_\ep}^{N_\ep} \frac{T_\ep^i}{i!} \exp\left(-i\left(\lambda\frac{(1-m)\ell}{6}-\La(\lambda)\right)\right),
\end{split}
\end{equation*}
for $\ep>0$ small enough since then $\frac{T_\ep}{i}\geq \frac{T_\ep}{N_\ep}=\frac{T_\ep}{[3T_\ep]+1}\geq \frac 1 6$. Hence we reach
\begin{equation*}
\begin{split}
&e^{-T_\ep}\sum_{i=M_\ep}^{N_\ep} \frac{T_\ep^i}{i!} R_i\left((1-m)L\right)
\leq e^{-T_\ep}\exp\left(T_\ep e^{-\left(\lambda\frac{(1-m)\ell}{6}-\La(\lambda)\right)}\right).
\end{split}
\end{equation*}
As a result the criterion  \eqref{estrb2} is satisfied as soon 
$$
e^{-\left(\lambda\frac{(1-m)\ell}{6}-\La(\lambda)\right)}\leq \frac{-\ln (4\theta)+\ln \ep+T_\ep}{T_\ep}.
$$
Recalling $T_\ep=\frac{1}{r^-}\ln\frac{2(\eta-\theta)}{\ep}$ this rewrites as
$$
e^{-\left(\lambda\frac{(1-m)\ell}{6}-\La(\lambda)\right)}\leq \frac{-\ln (4\theta)}{T_\ep}+1-r^-\frac{-\ln\ep}{\ln(2(\eta-\theta))-\ln\ep}.
$$
Since the right hand side tends to $1-r^->0$ as $\ep \to 0$,  one can choose $\ell>0$ large enough so that the above inequality holds for $\ep>0$ small enough. This completes the proof of Corollary \ref{prop:exp-decay}.
\end{proof} 

\medskip

Now, we denote $d_n$ the threshold sequence as given by Table 1 in Lemma \ref{LDH}, and $\gamma_n$ a sequence such that $\gamma _n \gg d_n$.  From the uniform conclusion \eqref{conclusion-table1} of Lemma \ref{LDH}, there is $\ep_0>0$ such that, for all $0<\ep<\ep_0$, all $L\geq \gamma_{M_\ep}$, 
$$
e^{-T_\ep}\sum_{i=M_\ep}^{N_\ep} \frac{T_\ep^i}{i!} R_i\left((1-m)L\right)\leq 
2e^{-T_\ep}\sum_{i=M_\ep}^{N_\ep} \frac{T_\ep^i}{i!}iR_1\left((1-m)L\right)\leq 
2 T_\ep R_1\left((1-m)L\right).
$$
As a consequence the criterion \eqref{estrb2} is asymptotically satisfied as soon as $L\geq L_\ep$ where
\begin{equation}\label{condition}
L_\ep \gg d_{M_\ep}\quad \text{ and }\quad
R_1\left((1-m)L_\ep\right)\leq C\frac{\ep}{\ln \frac{1}{\ep}},
\end{equation}
where  $C=C(\theta, r^-)$ is a positive constant and where we recall that $M_\ep=[\gamma^-T_\ep]$, $T_\ep=\frac{1}{r^-}\ln\frac{2(\eta-\theta)}{\ep}$. We now rely on Table 1 to compute $L_\ep$ for different kernels in Corollary \ref{prop:examples-be=2}.

\medskip

\begin{proof}[\bf Proof of Corollary \ref{prop:examples-be=2}-$(i)$] In that case, since $$
R_1\left((1-m)L_\ep\right)\le \frac{2c_2}{\alpha(1-m)^\alpha}\frac{1}{L_\ep^\alpha},
$$ 
the second condition in \eqref{condition} is asymptotically satisfied if 
$$(1-m)L_\ep=C^+\left(\frac1\ep\ln\frac 1\ep\right)^{\frac1\alpha}\text{ with appropriate } C^+=C^+(\theta, r^-, c_2, \alpha)>0. 
$$
In view of the line concerning $RV(\alpha)$  in Table 1 of Lemma \ref{LDH}, the first condition in \eqref{condition} is also satisfied. This proves Corollary \ref{prop:examples-be=2}-$(i)$. 
\end{proof}

\vspace{1ex}

\begin{proof}[\bf Proof of Corollary \ref{prop:examples-be=2}-$(ii)$] For $RV(\al)$ with $\al>2$ and any given $0<\tilde \alpha<\alpha$, it follows from \eqref{tail-RV} and known properties of slowly varying functions, see e.g. \cite[VIII, Lemma 2]{Feller}, that the second condition in \eqref{condition} is asymptotically satisfied if 
$$
L_\ep=\left(\frac1\ep\ln\frac1\ep\right)^{\frac{1}{\tilde \alpha}}.
$$ 
In view of the line concerning $RV(\alpha)$ in Table 1 of Lemma \ref{LDH}, the first condition in \eqref{condition} is also satisfied. This proves Corollary \ref{prop:examples-be=2}-$(ii)$.
\end{proof}
	
\medskip

\begin{proof}[\bf Proof of Corollary \ref{prop:examples-be=2}-$(iii)$] For $LN(\ga,\lambda,\rho)$ with $\ga>1$, $\lambda>0$, $\rho \in \R$, and given $0<\tilde \lambda<\lambda$, it follows from \eqref{tail-LN} that the second condition in \eqref{condition} is asymptotically satisfied if 
$$
L_\ep=  \exp \left[\left(\frac{1}{\tilde \lambda} \ln \frac 1\ep\right)^\frac{1}{\ga}\right]. 
$$
In view of the two  lines concerning $LN(\ga,\lambda,\rho)$ in Table 1 of Lemma \ref{LDH}, the first condition in \eqref{condition} is also satisfied. This proves Corollary \ref{prop:examples-be=2}-$(iii)$.
\end{proof}

\medskip
		
\begin{proof}[\bf  Proof of Corollary \ref{prop:examples-be=2}-$(iv)$] For $WE(\al,\la,\rho)$ with $0<\al<1$, $\la>0$, $\rho \in \R$, in view of \eqref{tail-WE} and the line concerning $WE(\al,\la,\rho)$ in Table 1 of Lemma \ref{LDH}, to reach the condition \eqref{condition} it is sufficient to have
\begin{equation}\label{condition2}
L_\ep\gg \left(\ln \frac 1 \ep\right)^{\frac{1}{2(1-\alpha)}} \quad \text{ and } \quad L_\ep^\rho e^{-\lambda (1-m)^\alpha L_\ep^\alpha} \leq C' \frac{\ep}{\ln \frac 1 \ep},
\end{equation}
for some appropriate $C'=C'(\theta,r^-,m,\alpha)$.

When $0<\alpha<\frac 23$, the choice 
$$
L_\ep=\frac{1}{1-m}\left(\frac{1}{\tilde \lambda}\ln \frac 1 \ep\right)^{\frac 1 \alpha}, 
$$
for any given $0<\tilde \lambda<\lambda$, ensures that both conditions in \eqref{condition2} are asymptotically satisfied. 

On the other hand, when $\frac  23 \leq \alpha <1$, the first condition in \eqref{condition2} prevents such a choice and we are compelled to take 
$$
L_\ep=\left(\ln \frac 1 \ep\right)^{\frac{1}{2(1-\tilde \alpha)}}, 
$$
for any given $\frac23\le\alpha<\tilde \alpha<1$, for the two conditions in \eqref{condition2} to be asymptotically satisfied. This proves Corollary \ref{prop:examples-be=2}-$(iv)$.
\end{proof}

\medskip 

Observe that when $\frac 2 3 \leq \alpha<1$ the above estimate on $L_\ep$ is not so good in particular when $\al\to1$. However for some specific forms of Weibull-like tails $WE(\al,\la,\rho)$ with $\rho=0$, we can rely on \eqref{WEbe=0} to obtain a sharper estimate.

\medskip

\begin{proof}[\bf Proof of Corollary \ref{prop:Weibull}] By a change of variable, the estimate for $WE(\al, \la, \rho=0)$ can be transformed into one of \eqref{tail-weibull-zero}, namely for $WE(\al, 1, \rho=0)$. From \eqref{tail-weibull-zero} and \eqref{WEbe=0}, we have, up to a multiplicative constant,
\begin{eqnarray}
e^{-T_\ep}\sum_{i=M_\ep}^{N_\ep} \frac{T_\ep^i}{i!} R_i\left((1-m)L\right) &  \leq  & e^{-T_\ep}\sum_{i=M_\ep}^{N_\ep} \frac{T_\ep^i}{i!} \left[\exp\left(-\frac{(1-m)^2}{20i}L^2\right)+ i \exp\left(-\frac{(1-m)^\al}{2^\al}L^\al\right)\right]\nonumber\\
& \leq  & \exp\left(-\frac{(1-m)^2}{20N_\ep}L^2\right)+ T_\ep \exp\left(-\frac{(1-m)^\al}{2^\al}L^\al\right).\nonumber
\end{eqnarray}
Recalling $N_\ep=N(T_\ep)=[3T_\ep]+1$ and $T_\ep=\frac{1}{r^-}\ln\frac{2(\eta-\theta)}{\ep}$, one can check that the above right hand side is asymptotically smaller than $\frac{\epsilon}{4\theta}$ if we choose 
$$
L=L_\ep:=\frac{C^+}{1-m}\left(\ln \frac 1 \ep\right)^{\frac 1 \alpha}
$$ 
with $C^+=C^+(J)>0$ large enough. As a result, the non-extinction criterion \eqref{estrb2} is satisfied. We have thus reached the announced sharper estimate when $\rho=0$ (which, when $\alpha \to 1$, is consistent with the exponential case shown in Corollary \ref{prop:exp-decay}). This completes the proof of Corollary \ref{prop:Weibull}.
\end{proof}

\section{Quantitative estimates of the threshold phenomena}\label{quantitative}

In this section, relying on Sections \ref{extinction} and \ref{non-extinction}, we complete the proof of the main results of  Section \ref{main results}. We denote $u_L^\ep=u^\ep_L(t, x)$ the solution to
    $$ 
    \partial _t u=J\ast u-u+f(u), 
    $$
    starting from $\phi_L^\ep(x)=(\theta+\ep)\1_{(-L, L)}(x)$. We start with the extinction results.
\medskip

\begin{proof}[Proof of Theorem \ref{ext}]
     By \eqref{r^+}, $u_L^\ep$ is a sub-solution to problem \eqref{super}. As a result, Theorem \ref{ext} follows from Proposition \ref{prop:extinction} and the comparison principle.
\end{proof}

\medskip

\begin{proof}[Proof of Corollary \ref{asmptotic1}] The proof is a rather straightforward combination of Corollary \ref{asmptotic} and the fact that  $ \max_{T>0}T e^{-r^+T}= e^{-1}/r^+$.
\end{proof}

\medskip

\begin{proof}[Proof of Corollary \ref{asmptotic2}]  The proof is a straightforward consequence of Corollary \ref{lighttail}.    
\end{proof} 

\medskip

\begin{proof}[Proof of Corollary \ref{ext:examples}] The proof is a straightforward consequence of Corollary \ref{remarkbeta=2} and Remark \ref{rem:as-easily}.
\end{proof}

\medskip

Let us now investigate the propagation results.

\medskip

\begin{proof}[Proof of Theorem \ref{theo-propagation}]
    From \eqref{r-moins} in Assumption \ref{ass:f}, one has
    $$ 
    f(w)\geq \tilde{g}(w),\quad\forall w\in (-\infty, \de], 
    $$
    where $\tilde{g}(w)=r^-(w-\theta)$ was defined in \eqref{g}. Let $0<\alpha<\frac{\de-\theta}2$ be given small enough so that we can define a Lipschitz continuous function $\tilde{f}: \R\to\R$ such that $\tilde f\leq f$,
    \begin{eqnarray}
    \tilde{f}(w)=\begin{cases}
    f(w)\quad &\text{for } w\in (-\infty, \theta) \cup [\de, \infty),\\
    r^-(w-\theta)\quad &\text{for } w\in [\theta, \theta+2\al],
    \end{cases}
    \end{eqnarray}
and $\tilde f$ satisfies Assumption \ref{ass:f} (in particular $\int_0^1\tilde f(s)ds>0$ and \eqref{r-moins} holds on $[0,\theta+2\alpha]$). Denote $\tilde{w}=\tilde{w}(t, x)$ the solution to 
    \begin{equation}\label{tildew}
    \partial _t \tilde{w}=J\ast \tilde{w}-\tilde{w}+\tilde{f}(\tilde{w}),
    \end{equation} 
starting from $ \tilde{w}(0, x)=\phi _L ^\ep(x)$, so that $\tilde{w}(t, x)\leq u_L^\ep(t, x)$ from the comparison principle. Consider the time $T_\ep=\frac{1}{r^-}\ln\frac{2\al}{\ep}$. For $0<\ep\leq 2 \alpha$ we know from Proposition \ref{nonextinction} $(i)$ (setting $\eta=\theta+\al$) that $\tilde{w}\leq\theta+2\al\le\de$ on $[0, T_\ep]\times\R$. Since
    $$
    \tilde{f}(w)\geq \tilde g (w), \quad\forall w\in(-\infty, \theta+2\al], 
    $$    
   one obtains from Proposition \ref{nonextinction} $(ii)$ that, for any given $m\in(0, 1)$, there exists $\ep_0>0$ such that for all $\ep\in(0, \ep_0)$ and $L>L_\ep$, where $L_\ep$ satisfies \eqref{nonextinctioncondition}, one has
    $$ 
    u_L^\ep(T_\ep,x)\geq \tilde{w}(T_\ep, x)\geq(\theta+\al)\1_{(-mL_\ep,mL_\ep)}. 
    $$
From the propagation Assumption \ref{ass:prop} for the nonlinearity $f$, we know that $L_\alpha^{prop}<+\infty$ exists, that is, for $\ell > L_\alpha^{prop}$ the solution to \eqref{nonlocal} starting from $(\theta+\alpha)\1_{(-\ell,\ell)}$ propagates. As a result, for $\ep>0$ small enough so that $mL_\ep> L_\alpha^{prop}$, one has $u_L^\ep(T_\ep+t, x)\to 1$ as $t\to+\infty$ locally uniformly in space and therefore $u_L^\ep(t, x)\to1$ as $t\to+\infty$ locally uniformly in space. We have thus proved that, for $L_\ep$ satisfying \eqref{nonextinctioncondition}, we have $L_\ep^{prop}\leq L_\ep$ for $\ep>0$ small enough. This completes the proof.
\end{proof}    

\medskip

\begin{proof}[Proof of Corollary \ref{prop:corollary1}]
	Combining the arguments in the proof of Theorem \ref{theo-propagation} with Corollary \ref{prop:0<beta<2}, one can obtain the desired results. 
\end{proof}

\medskip

\begin{proof}[Proof of Corollary \ref{prop:corollary2}]
	Combining the arguments in the proof of Theorem \ref{theo-propagation} with Corollary \ref{prop:exp-decay}, one can obtain the desired results. 
\end{proof}

\medskip

\begin{proof}[Proof of Corollary \ref{prop:corollary3}]
	Combining the arguments in the proof of Theorem \ref{theo-propagation} with Corollary \ref{prop:examples-be=2} and Corollary \ref{prop:Weibull}, one can obtain the desired results. 
\end{proof}

\section{Propagation threshold}\label{propagation}

This section is devoted to the proof of the propagation threshold result, namely Theorem \ref{theo-ass-prop}. We shall rely on some ideas developed by 
Fife and McLeod in \cite[Theorem 3.2]{fife1977approach} and  crucially make use of the following integrability properties of the wave profile $U$.

\begin{lemma}\label{integrability}
	Let Assumption \ref{TWA} be satisfied. Then any monotone traveling wave $(c,U)$ solving \eqref{TW}, and whose existence follows from \cite{bates1997traveling},
satisfies the integrability properties
	$$ 
	\int_{-\infty}^0 U(x)dx<+\infty \quad \text{ and }	\quad \int_0^{+ \infty}(1-U(x))dx<+\infty.
	$$
\end{lemma}

\begin{proof} This is nothing else than \cite[(5.9)]{bates1997traveling} but, to enlight the importance of the finite first moment hypothesis \eqref{first_moment} and for the convenience of the reader, we give a proof of the integrability of $1-U$ in $+\infty$ (the other one being similar).  

	From the equation satisfied by $U$, $U(+\infty)=1$, and the assumption $f'(1)<0$, there are $R>0$ and $C_1>0$ such that  
	\begin{equation}\label{ineqq}
	-(J\ast U)(x)+U(x)+cU'(x)=f(U(x))\ge C_1(1-U(x))>0,\quad \forall x\geq R.
	\end{equation}
Now we consider  $\ell>0$. When $c\neq 0$ the regularity of the wave, see subsection \ref{ss:towards}, enables to write
	\begin{eqnarray*}
		\int_R^{R+\ell} (J\ast U(x)-U(x))dx&=&\int_R^{R+\ell} \int_{-\infty}^{+\infty} J(y)[U(x-y)-U(x)] \,dy dx\nonumber\\
		&=&-\int_R^{R+\ell} \int_{-\infty}^{+\infty} J(y)\int_0^1 yU'(x-sy) \, ds dy dx\nonumber\\
		&=&-\int_{-\infty}^{+\infty} yJ(y) \int_0^1\left( \int_{R}^{R+\ell} U'(x-sy) dx\right) ds dy\nonumber\\
		&=&-\int_{-\infty}^{+\infty} yJ(y)\int_0^1\left(U(R+\ell-sy)-U(R-sy)\right) ds dy,\nonumber
	\end{eqnarray*}	
thanks to Fubini's theorem. Note that, when $c=0$, a mollifying argument as in \cite[Lemma 3.2]{alfaro2017propagation} shows that the above conclusion is still valid. As a result, since $0\leq U\leq 1$, we get, for any $\ell>0$,
	\begin{equation*}
	\left|\int_R^{R+\ell} (J\ast U(x)-U(x))dx\right| \leq 2\int_{-\infty}^{+\infty}|y|J(y)dy.
	\end{equation*}
	On the other hand, for any $\ell>0$,
	$$
	\left|\int_R^{R+\ell} cU'(x) dx \right|=  \vert c\vert  \left( U(R+\ell)-U(R) \right) \le \vert c\vert.
	$$
We thus deduce from \eqref{ineqq} that, for any $\ell>0$,
	$$
	C_1\int_R^{R+\ell}(1-U(x))dx \le \vert c\vert +2\int_{-\infty}^{+\infty}|y|J(y)dy <+\infty,
	$$
	which implies that $1-U\in L^1(R,\infty)$ and completes the proof of the lemma.
\end{proof}

\medskip

We now turn to the proof of Theorem \ref{theo-ass-prop}. 

\medskip

\begin{proof}[Proof of Theorem \ref{theo-ass-prop}] Let Assumption \ref{TWA} be satisfied and assume $\int_0^1f(u)du>0$. Consider $(c,U)$  an increasing traveling wave solution to \eqref{nonlocal} with $c>0$. We aim at showing that, for any $\ep\in(0,1-\theta]$, there exists $L>0$ large enough such that propagation occurs for $u_L^\ep$ the solution to \eqref{nonlocal} starting from \eqref{initial}. To do so and as mentioned above, we rely on the approach of \cite{fife1977approach} for the classical diffusion case.
	
	Let us consider the function $\underline{u}$ given by
	$$ 
	\underline{u}(t, x):=U_+(t, x)+U_-(t, x)-1-q(t), 
	$$
	with $U_\pm(t, x):=U(\zeta_\pm(t,x))$, where $\zeta_\pm(t,x)$ take the form
	\begin{equation*}
	\zeta_+(t,x)=x+ct-\xi(t),\;\; \zeta_-(t,x)=\zeta_+(t,-x)=-x+ct-\xi(t).
	\end{equation*}
Here $q=q(t)$ and $\xi=\xi(t)$ are functions to be determined for $\underline{u}$ to be a sub-solution to \eqref{nonlocal}.

From the above and the $U$-equation, we straightforwardly compute, for $t>0$ and $x\in \R$, 
	\begin{eqnarray}\label{sub2}
	N\underline{u}(t,x) &:=& \partial_t \underline{u}(t,x)-J\ast\underline{u}(t,x)+\underline{u}(t,x)-f\left(\underline{u}(t,x)\right)\nonumber \\
	&{}=&-\xi'(t)\left[U'(\zeta_+(t,x))+U'(\zeta_-(t,x))\right]\nonumber\\
	&&+f\big(U_+(t,x)\big)+f\big(U_-(t,x)\big)-f\Big(U_+(t,x)+U_-(t,x)-1-q(t)\Big)-q'(t).
	\end{eqnarray}
	
	Before going further, let us introduce some notations. Denote $\alpha:=\theta+\ep\in (\theta,1]$ the fixed height of the step initial data. Fix two constants $1-\alpha<q_0<q_1<1-\theta$  so that
	$$ 
	\theta<1-q_1<1-q_0<\al, 
	$$
	and define the function $\Phi$, continuous on $\R\times [0,+\infty)$, as
	$$ 
	\Phi(u, s):=
	\begin{cases}
	\frac{f(u-s)-f(u)}s, \quad &\text{ if } s>0,\\
	-f'(u), \quad &\text{ if  } s=0.
	\end{cases} 
	$$
Moreover, for $0<s\leq q_1$ we have $\theta<1-q_1\leq1-s<1$, so that $\Phi(1, s)>0$. Also $\Phi(1, 0)=-f'(1)>0$. Thus there exists $\mu>0$ such that $\Phi(1, s)\geq 2\mu$ for $0\leq s\leq q_1$. By continuity, there exists a $\de>0$ such that 
	$\Phi(u, s)\geq\mu$  for $1-\de\leq u\leq1$ and $0\leq s\leq q_1$. It then follows that 
	\begin{equation}\label{f'(1)<0}
	f(u-s)-f(u)\geq \mu s, \quad \text{ for all  } 1-\de\leq u\leq1 \text{ and } 0\leq s\leq q_1.
	\end{equation} 
Last, we fix $b>0$ large enough so that
	\begin{equation}\label{def-b}
	f(u)\leq b(1-u), \quad \text{ for all } 0\leq u \leq 1.
	\end{equation}

	\begin{claim}\label{limitbehavior}
		For any $s_0\geq 0$,
		$$
		g(t):= \int_{0}^{t}e^{-\mu(t-s)}(1-U(cs+s_0))ds
		$$
		tends to $0$ as $t\to+\infty$.
	\end{claim}	
	
	\begin{proof}
		Observe first that $0<g(t)\leq\int_0^t e^{-\mu(t-s)} ds=\frac{1}{\mu}\left(1-e^{-\mu t}\right)$, so that $g\in L^\infty(0, \infty)$. Next note that 
		$g'(t)=1-U(ct+s_0)-\mu g(t)$, so that $g'\in L^\infty(0, \infty)$. Next, by Fubini-Tonelli's theorem, 
		\begin{eqnarray}
		\int_0^{+\infty} g(t) dt &=& \int_0^{+\infty} \int_0^t e^{-\mu(t-s)}(1-U(cs+s_0)) ds dt = \int_0^{+\infty} e^{\mu s} (1-U(cs+s_0)) \int_s^{+\infty} e^{-\mu t} dt ds\nonumber\\
		&=&\frac{1}{\mu}\int_0^{+\infty}(1-U(cs+s_0)) ds=\frac{1}{c\mu}\int_{s_0}^{+\infty}(1-U(x))dx\nonumber\\
		&\le&\frac{1}{c\mu}\int_{0}^{+\infty}(1-U(x))dx=\frac{1}{c\mu}\norm{1-U}_{L^1(0, \infty)}<+\infty,\label{g(t)}
		\end{eqnarray}
		from Lemma \ref{integrability}, so that  $g\in L^1(0, \infty)$. Now, the combination of $g\in L^1(0, \infty)$ and $g'\in L^\infty(0, \infty)$ enforces $g(t)\to 0$ as $t\to+\infty$, which completes the proof of the claim.
	\end{proof}

\medskip

From $g(0)=0$ and the above claim, $g$ attains its maximum at some $t_0>0$, and
	\begin{equation}\label{g(t_0)}
	g(t_0)=\max_{t\geq 0}g(t)=\frac{1}{\mu}(1-U(ct_0+s_0))\le \frac{1}{\mu}(1-U(s_0)). 
	\end{equation}
	 For constants $\xi_0<0$ and $\eta_0>0$ with $s_0:=-\eta_0-\xi_0>0$ to be be determined below, we select
	\begin{equation}\label{q}
	q(t):=q_0e^{-\mu t} + bg(t)=q_0e^{-\mu t} + b\int_{0}^{t}e^{-\mu(t-s)}(1-U(cs+s_0))ds,\quad t \geq 0.
	\end{equation}
	We also let $\xi(t)=\xi_0+\eta(t)$ where $\eta$ is to be selected below with the properties
	\begin{equation}\label{eta-to-be-selected}
	\eta(0)=0, \; \eta'(t)> 0,\; \eta(t)\leq \eta_0\leq -\xi_0.
	\end{equation}
	
In the sequel, we aim at reaching $N\underline u(t,x)\leq 0$ for all $x\in \R$, $t> 0$. Since both $\underline u(t,\cdot)$ and $J$ are symmetric, it is sufficient to work with $x\geq 0$. Since $U'>0$ we have, for all $x\geq 0$ and $t> 0$, 
	\begin{eqnarray}
	1-U_+(t,x)+q(t) &=& 1-U(x+ct-\xi(t))+q(t) \nonumber\\
	&\le& 1-U(ct-\xi_0-\eta(t))+q_0+bg(t_0)\nonumber\\
	&\leq& 1-U(-\xi_0-\eta_0)+q_0+\frac{b}{\mu}(1-U(-\xi_0-\eta_0))\nonumber\\
	&=&(1+b/\mu)(1-U(s_0))+q_0.\nonumber
	\end{eqnarray}
	Choose $s_0>0$ large enough so that	
	$(1+b/\mu)(1-U(s_0))+q_0\leq q_1$.	
	As a consequence, for any such choice, one has, for all $t\geq 0$ and $x\geq 0$,
	\begin{equation}\label{q_1}
	0 \le 1-U_+(t,x)+q(t) \le q_1.
	\end{equation}
	Below we complete the construction of the sub-solution by investigating the sign of $N\underline u(t,x)$ for $x\geq 0$ and $t> 0$. To do so, recalling that $\de>0$ was chosen above for \eqref{f'(1)<0} to hold,  we split our analysis according to the value of $U_-(t,x)$. 
	
\medskip	
	
	\noindent \textbf{First case: $1-\de\le U_-(t,x) \le 1$.} Then, from \eqref{f'(1)<0} and \eqref{q_1},
	\begin{equation}\label{truc}
f\big(U_-(t,x)\big)-f\Big(U_-(t,x)-(1-U_+(t,x)+q(t))\Big)
\leq-\mu(1-U_+(t,x)+q(t)).
\end{equation}
Plugging this into \eqref{sub2}, using $U'>0$, $\xi'(t)>0$ and  \eqref{def-b}, we reach
	\begin{eqnarray*}
		N\underline{u}(t,x) &\le& -\mu(1-U_+(t,x)+q(t))+b(1-U_+(t,x))-q'(t)\nonumber\\
		&=&(b-\mu)(1-U_+(t,x))-\mu q(t)-q'(t)\nonumber\\
		&\le & b(1-U_+(t,x))-\mu q(t)-q'(t)\nonumber\\
		&=& b(1-U(x+ct-\xi(t)))-\mu q(t)-q'(t)\nonumber\\
	&=& b(1-U(x+ct-\xi(t)))-b(1-U(ct-\xi_0-\eta_0))	
	\end{eqnarray*}
 from the definition of $q(t)$ in \eqref{q}. Since $U'>0$, $x\geq 0$ and $-\xi(t)\geq -\xi_0-\eta_0$, we end up with $N\underline u(t,x)\leq 0$.

\medskip
	
	\noindent \textbf{Second case: $0\le U_-(t,x)\le \de$.} Let us recall that 
	$f\in C^1(\R)$ and $f'(0)<0$. Therefore, up to modify $f$ on $(-\infty,0)$ (which is harmless for the problem under consideration since solutions are nonnegative), we may assume that there are $\tilde{\mu}>0$ and $\tilde{\de}>0$ such that
\begin{equation}\label{concave}
f'(u)\le -\tilde{\mu},\quad  \forall u\in(-\infty, \tilde{\de}].
\end{equation}
Also, up to reducing $\mu$ and $\delta$ appearing in \eqref{f'(1)<0} if necessary, we may assume $0<\de\le\tilde{\de}$ and $0<\mu\le\tilde{\mu}$. As a result,
	\begin{equation*}
	f(u)-f(u-s)=\int_{u-s}^u f'(\sigma)d\sigma \leq -\mu s, \quad  \text{ for all } -\infty< u\leq \de  \text{ and } s\ge 0.
	\end{equation*}
	From this we, again, deduce \eqref{truc} and conclude as in the first case.
		
	\medskip

	\noindent \textbf{Third case: $\de\leq U_-(t,x)\leq 1-\de$.} If we denote $C>0$ the Lipschitz constant of $f$ on the interval $[\delta-q_1,1-\delta]$, we deduce from $\de\leq U_-(t,x)\leq 1-\de$ and \eqref{q_1} that
\begin{equation}\label{truc-bis}
f\big(U_-(t,x)\big)-f\Big(U_-(t,x)-(1-U_+(t,x)+q(t))\Big)
\leq C(1-U_+(t,x)+q(t)).
\end{equation}
From \eqref{def-b}, we have $f(U_+(t,x))\leq b(1-U_+(t,x))$. Moreover, in this third case, we have
\begin{equation}\label{middle}
	U'(\zeta_+(t,x))+U'(\zeta_-(t,x))\geq U'(\zeta_-(t,x))\geq \min_{U^{-1}(\delta)\leq z\leq U^{-1}(1-\delta)} U'(z) := \vartheta>0. 
	\end{equation}
Plugging this into \eqref{sub2}, we get
		\begin{eqnarray*}
	N\underline{u}(t,x)&\le& -\vartheta\xi'(t)+(C+b)(1-U_+(t,x))+Cq(t)-q'(t)\nonumber\\
	&=& -\vartheta\eta'(t) +C(1-U_+(t,x))+(C+\mu)q(t)+b\left(U(ct+s_0)-U(\zeta_+(t,x)\right)
	\end{eqnarray*}	
from computing $q'(t)$.	Since $\zeta_+(t,x)=x+ct-\xi(t)\geq ct-\xi_0-\eta_0=ct+s_0$ and $U'>0$, we obtain
	\begin{equation}
	\label{sub3}
	N\underline{u}(t,x)\leq  -\vartheta\eta'(t) +C\left(1-U(x+ct-\xi(t))\right)+(C+\mu)q(t).
	\end{equation}
	We now  select
	\begin{eqnarray}
	\eta(t) &:=& \frac{C}{\vartheta}\int_{0}^{t}(1-U(cs+s_0))ds+\frac{(C+\mu)q_0}{\vartheta}\int_{0}^{t}e^{-\mu s}ds\nonumber\\
	&& +\frac{b(C+\mu)}{\vartheta}\int_{0}^{t}\int_{0}^{s}e^{-\mu(s-\tau)}(1-U(c\tau+s_0))d\tau ds.\nonumber
	\end{eqnarray} 
	Obviously $\eta(0)=0$ and 
	\begin{equation}\label{vartheta-prime}
	\vartheta \eta'(t)=C\left(1-U(ct+s_0)\right)+(C+\mu)q(t)>0,
	\end{equation}
	and thus $\eta(t)\leq \eta(+\infty)$ for all $t\geq 0$. We estimate $\eta(+\infty)$ as follows:
	\begin{eqnarray*}
	\eta(+\infty)&=& \frac{C}{\vartheta}\int_{0}^{+\infty}(1-U(cs+s_0))ds+\frac{(C+\mu)q_0}{\vartheta}\int_{0}^{+\infty}e^{-\mu s}ds\nonumber\\
	&& +\frac{b(C+\mu)}{\vartheta}\int_{0}^{+\infty}\int_{0}^{s}e^{-\mu(s-\tau)}(1-U(c\tau+s_0))d\tau ds\\
	&\le& \frac{C}{c\vartheta}\int_0^{+\infty}(1-U(x))dx+\frac{(C+\mu)q_0}{\vartheta\mu}+\frac{b(C+\mu)}{c\vartheta\mu}\int_0^{+\infty}(1-U(x))dx=:\eta_0,\nonumber
	\end{eqnarray*}
	from the same computation as in \eqref{g(t)}. 
	
	Plugging \eqref{vartheta-prime} into \eqref{sub3}, we reach $
	N\underline{u}(t,x) \leq C\left(U(ct+s_0)-U(x+ct-\xi(t)\right)
$ which is nonpositive as already argued above. 
	
	\medskip

	\noindent \textbf{Conclusion.}  With the above choices, we have therefore verified that $N\underline{u}(t, x)\leq0$ for all $(t, x)\in(0, +\infty)\times\R$.  For $|x|\le L$, one has
	$$ 
	\underline{u}(0, x)=U(x-\xi_0)+U(-x-\xi_0)-1-q_0< 1-q_0<\al=\al\1_{(-L, L)}(x). 
	$$
For $\vert x\vert \geq L$, one has	
	$$
	\underline{u}(0, x)=U(x-\xi_0)+U(-x-\xi_0)-1-q_0\leq U(-L-\xi_0)-q_0<0,
	$$
	if $L=L(\xi_0)>0$ is large enough. As a result, for such a  large $L>0$,
	$$ 
	\underline{u}(0, x)\leq\al\1_{(-L, L)}(x),\quad \forall x\in\R.
	$$  
It follows from the comparison principle that 
$u(t,x)\geq \underline u(t,x)$ for all $t\geq 0$, $x\in \R$. Since $\underline u$ satisfies \eqref{qqch}, so does $u$ and the proof is complete. 
\end{proof}

\appendix

\section{Appendix: the uniformity in \eqref{cramer}}\label{s:appendix}

In this Appendix, we show that the uniformity of the limit \eqref{cramer} is a consequence of an important estimate taken  from \cite{hoglund}. 

\medskip

Let $J:\R\to\R$ be a non trivial and nonnegative function such that
\begin{itemize}
\item[$(i)$] $J\in L^1(\R)$, $J(-x)=J(x)$ a.e. $x\in\R$;

\item[$(ii)$] there exists $\alpha_0>0$ such that $x\mapsto J(x)e^{\alpha_0x}\in L^1(\R)$.
\end{itemize}
In particular,  $\int_\R J(x)e^{\alpha x}dx<+\infty$ for all $\alpha\in [-\alpha_0,\alpha_0]$. For $\alpha\in A:=(-\alpha_0,\alpha_0)$, consider $J_\alpha:\R\to\R$ given by
$$
J_\alpha(x):=\frac{e^{\alpha x} J(x)}{\int_\R e^{\alpha y} J(y)dy}.
$$
Note, that for all $\alpha\in A$, $J_\alpha$ is a probability distribution which admits moments at any orders. Then consider $m:A\to \R$ and $v:A\to \R$ where, for any $\alpha\in A$, $m(\alpha)$ and $v(\alpha)$ denote the mean value and the variance of $J_\alpha$, respectively. Note that these two maps are smooth and that
\begin{equation*}
m'(\alpha)=v(\alpha)>0,\;\forall \alpha\in A.
\end{equation*}
Moreover due to $(i)$, one has $m(0)=0$. Define also the function $a:m(A)\to \R$ by $a=m^{-1}$, so that $0\in m(A)$, $a(0)=0$ and $a$ is continuous and increasing. Using the above notations, \cite[Theorem A $(i)$]{hoglund} reads as follows.

\begin{theorem}\label{THEO-app}
Set for $n\geq 0$ and $L>0$,
$$
R_n(L):=\int _{|x|\geq L} J^{*(n)}(x)dx.
$$
Then, under the above assumptions, for any compact set $K$ of $A\cap [0,\infty)$ one has
\begin{equation*}
R_n(nx)=2e^{-n\Lambda^*(x)}\tau\left(nv(a(x))a(x)^2\right)\left(1+o(1)\right)\text{ as $n\to+\infty$},
\end{equation*}
uniformly for $x\in \R$ such that $a(x)\in K$. Herein $\tau:[0,\infty)\to\R$ is the function given by
$$
\tau(\lambda)=(2\pi)^{-1/2}e^{\lambda/2}\int_{\sqrt \lambda}^{+\infty}e^{-y^2/2}dy,
$$
while $\Lambda^*$ denotes the Fenchel-Legendre transform of the logarithmic moment generating function of $J$, as defined in Lemma \ref{lem:expo}.
\end{theorem}

 Equipped with this we can conclude on the uniformity of the limit \eqref{cramer} as stated in Lemma \ref{lem:expo}.
 
 \medskip

\begin{proof} Since $a$ is continuous, increasing and $a(0)=0$, we can fix $L_1>0$ such that
$$
0<L_1<\lim_{\alpha\to \alpha_0}m(\alpha),\; \text{ and } \; a(L_1)<\alpha_0.
$$
Now select $L_0>0$ such that $0<L_0<L_1$.
With such a choice, $ a\left([L_0,L_1]\right)$ is a compact subset of $A\cap [0,\infty)$. Hence from Theorem \ref{THEO-app} we obtain, uniformly for $x\in [L_0,L_1]$,
\begin{equation*}
\frac{1}{n}\ln R_n(nx)=-\Lambda^*(x)+\frac{1}{n} \ln \tau\left(nv(a(x))a(x)^2\right)+O\left(\frac{1}{n}\right), \quad \text{ as $n\to+\infty$}.
\end{equation*}
Note note that $0<a(L_0)\leq a(x)\leq a(L_1)<\alpha_0$ for all $x\in [L_0,L_1]$, so that there exists $\eta\in (0,1)$ such that
$$
\eta \leq v(a(x))a(x)^2\leq \eta^{-1},\;\forall x\in [L_0,L_1].
$$
Hence, since one has
$$
\tau(\lambda)\sim (2\pi\lambda)^{-1/2},\quad  \text{ as } \lambda\to +\infty,
$$
this yields
$$
\frac{1}{n} \ln \tau\left(nv(a(x))a(x)^2\right)=O\left(\frac{\ln n}{n}\right),
$$
as $n\to+\infty$ and uniformly with respect to $x\in [L_0,L_1]$.
This proves the uniformity in \eqref{cramer}.
 \end{proof}

\bigskip

 \noindent{\bf Acknowledgements.}  M. Alfaro is supported by the ANR project DEEV ANR-20-CE40-0011-01. H. Kang would like to acknowledge the region Normandie for the financial support of his postdoc.

\bibliography{hulk}
\bibliographystyle{plain}

\end{document}